\documentclass[fleq]{amsart}
\usepackage{amssymb,amsmath,latexsym,amsbsy,color,enumerate}
\numberwithin{equation}{section}

\newtheorem{theorem}{Theorem}
\newtheorem{definition}[theorem]{Definition}
\newtheorem{lemma}{Lemma}
\newtheorem{remark}{Remark}
\newtheorem{proposition}{Proposition}

\begin{document}
 \title{Some dynamical properties of pseudo-automorphisms in dimension $3$}
 \author{Tuyen Trung Truong}
       \address{Department of Mathematics, Syracuse University, Syracuse, NY 13244, USA}
 \email{tutruong@syr.edu}
 
\thanks{}
    \date{\today}
    \begin{abstract}
Let $X$ be a compact K\"ahler manifold of dimension $3$ and let $f:X\rightarrow X$ be a pseudo-automorphism. Under the mild condition that $\lambda _1(f)^2>\lambda _2(f)$, we  prove the existence of invariant positive closed $(1,1)$ and $(2,2)$ currents, and we also discuss the (still open) problem of  intersection of such currents. We prove a weak equidistribution result (which is essentially known in the literature) for Green $(1,1)$ currents of meromorphic selfmaps, not necessarily $1$-algebraic stable, of a compact K\"ahler manifold of arbitrary dimension; and discuss how a stronger equidistribution result may be proved for pseudo-automorphisms in dimension $3$. As a byproduct, we show that the intersection of some dynamically related currents are well-defined with respect to our definition here, even though not obviously to be seen so using the usual criteria.  
\end{abstract}
\maketitle
\section{Introduction}       
\label{SectionIntroduction} 

This paper studies the dynamics of pseudo-automorphisms in dimension $3$.  Let $X$ be a compact K\"ahler manifold of dimension $3$. A map $f:X\rightarrow X$ is a pseudo-automorphism if it is a bimeromorphic map so that both $f$ and $f^{-1}$ has no exceptional hypersurfaces (see Dolgachev-Ortland \cite{dolgachev-ortland}). Note that (see Lemma \ref{LemmaImageOfCurveByPseudoAutomorphisms}) if $f$ is a pseudo-automorphism then so are the iterates $f^n$ ($n\in \mathbb{Z}$). Recent constructions by Bedford-Kim \cite{bedford-kim}, Perroni-Zhang \cite{perroni-zhang}, Oguiso \cite{oguiso1}\cite{oguiso2}, and Blanc \cite{blanc} provided many interesting examples of such maps. Among bimeromorphic selfmaps, it may be argued that the class of pseudo-automorphisms is the second best after that of automorphisms. In dimension $2$, pseudo-automorphisms are automorphisms. 

One of the difficulties when studying dynamics of meromorphic maps in dimension $>2$ is that in general we can not pull back positive closed currents of bidegree $>(1,1)$. Our first main result shows that it is possible to do so for pseudo-automorphims in dimension $3$. For $0\leq p\leq 3$, let $\mathcal{D}^{p}(X)$ be the real vector space generated by positive closed $(p,p)$ currents on $X$, and let $DSH ^{p}(X)$ be the space of $DSH$ $(p,p)$ currents on $X$ (for precise definitions of these classes and their properties, see Section 2). By definition it follows that $\mathcal{D}^{p}(X)\subset DSH ^{p}(X)$. For a closed current $T$ we let $\{T\}$ denote its cohomology class.

Theorem \ref{TheoremPullbackPushforwardCurrents} below shows the possibility of pulling back or pushing forward currents by $f$, and proves the compatibility of such operators with the iteration. The definition of pulling back or pushing forward current we use here was developed in our previous paper \cite{truong1} and is refined in the current paper.   

\begin{theorem}
Let $X$ be a compact K\"ahler manifold of dimension $3$ and let $f:X\rightarrow X$ be a pseudo-automorphism. Then, with respect to Definitions  \ref{DefinitionPullbackCurrentsByMeromorphicMaps} and \ref{DefinitionPullbackDdcOfOrderSCurrents}, for any $n\in \mathbb{Z}$ there are well-defined pullback and pushforward operators  $(f^n)^*$ and $(f^n)_*$ from each of the spaces $\mathcal{D}^{1}(X)$, $DSH^{1}(X)$, and $\mathcal{D}^2(X)$ into itself. These operators are continuous with respect to the topologies on the corresponding spaces, and hence are compatible with the pullbacks or pushforwards on cohomology groups. Moreover, these operators are compatible with iteration in the sense that $(f^n)^*=(f^*)^n$ and $(f^n)_*=(f_*)^n$ for any $n\in\mathbb{Z}$. 
\label{TheoremPullbackPushforwardCurrents}\end{theorem}
Remark: In the case $X=\mathbb{P}^k$ a projective space, results similar to Theorem \ref{TheoremPullbackPushforwardCurrents} were proved in Dinh-Sibony \cite{dinh-sibony4} and deThelin-deVigny \cite{deThelin-deVigny}, using super-potential theory. For other manifolds, previously there were no such results. 

Next we discuss the existence of invariant positive closed currents for $f$. Since the map $f^*:\mathcal{D}^{1}(X)\rightarrow \mathcal{D}^1(X)$ preserves the cone of positive closed $(1,1)$ currents, it follows by a Perron-Frobenius type argument that $\lambda _1(f)$ is an eigenvalue of $f^*:H^{1,1}(X)\rightarrow H^{1,1}(X)$. We have the following

\begin{theorem}
Let $X$ be a compact K\"ahler manifold of dimension $3$, and $f:X\rightarrow X$ a pseudo-automorphism. Assume that $\lambda _1(f)^2>\lambda _2(f)$. Then  

a) There is a non-zero positive closed $(1,1)$ current $T^+$ such that $f^*(T^+)=\lambda _1(f)T^+$. Moreover, $T^+$ has no mass on hypersurfaces.   

b) There is a non-zero positive closed $(2,2)$ current $T^-$ such that $f_*(T^-)=\lambda _1(f)T^-$.

c) We can choose $T^+$ and $T^-$ such that in cohomology $\{T^+\}.\{T^-\}=1$. 
\label{TheoremInvariantCurrents}\end{theorem}
Remarks: 

1) The assumption that $\lambda _1(f)^2>\lambda _2(f)$ is not a real restriction. In fact, when $\lambda _1(f)>1$ this condition is satisfied for either the map $f$ or its inverse $f^{-1}$. 

2) Part a) of Theorem \ref{TheoremInvariantCurrents} is already known in the literature, however its refinement in Theorem \ref{TheoremGreenCurrents} below seems to be new. 

Compared with the results for meromorphic maps in dimension $2$ (see e.g. Diller-Favre \cite{diller-favre}, Diller-Dujardin-Guedj \cite{diller-dujardin-guedj1}), for automorphisms in any dimension (see e.g. Cantat \cite{cantat}, Dinh-Sibony \cite{dinh-sibony3}\cite{dinh-sibony4}\cite{dinh-sibony5}), for Green $(1,1)$ currents of meromorphic maps whose invariant cohomology class satisfying several conditions (see Sibony \cite{sibony}, Diller-Guedj \cite{diller-guedj}, Guedj \cite{guedj3}, Bayraktar \cite{bayraktar}), and for linear fractional maps (see Bedford-Kim \cite{bedford-kim}) we are led to the following natural questions:

{\bf Question 1.} Does $T^+$ in Theorem \ref{TheoremInvariantCurrents} satisfy an equi-distribution property, i.e. for every smooth closed $(1,1)$ form $\theta$ of the same cohomology class as $T^+$ we have
\begin{eqnarray*}
\lim _{n\rightarrow\infty}\frac{(f^n)^*(\theta )}{\lambda _1(f)^n}=T^+?    
\end{eqnarray*}
If this equi-distribution property holds, does it also hold for any smooth closed $(1,1)$ form $\theta$ for which $\int _X\theta \wedge T^-=1$? What about similar questions for $T^-$?

We can answer Question 1 in affirmative under an additional condition, whose proof will be given in Section \ref{SectionDiscussion} where we also discuss some other cases where the same idea may apply. 

\begin{theorem}
Let $X$ be a projective manifold, and let $f:X\rightarrow X$ be a dominant meromorphic map which is $1$-algebraic stable such that $\lambda _1(f)^2>\lambda _2(f)$. Assume that $f$ is holomorphic-like, i.e. it satisfies the following two conditions (i) for the eigenvector $\{\theta \}\in H^{1,1}(X)$ of eigenvalue $\lambda _1(f)$ we have $\{\theta \}.\{\theta \}=0$ and (ii) there is a desingularization $Z$ of the graph of $f$ such that the induced projection to the first factor $\pi :Z\rightarrow X$ is a composition of blowups along  smooth centers for which if $E\subset Z$ is a hypersurface then $dim (\pi (E))\geq dim(X)-2$. Then for any smooth closed $(1,1)$ form $\theta$ in the cohomology class of $\{\theta\}$ the limit
\begin{eqnarray*}
\lim _{n\rightarrow\infty}\frac{(f^*)^n(\theta )}{\lambda _1(f)^n}
\end{eqnarray*}
exists, and is the same positive closed current given in Theorem \ref{TheoremGreenCurrents} below.
\label{TheoremPositivityOfGreenCurrentsPseudoAutomorphism}\end{theorem}
Note that in the case $f$ is a holomorphism then the two conditions (i) and (ii) are automatically satisfied, and in this case the result is known in the literature. While condition (i) or some variant of it seems essential we feel that condition (ii) is not needed, see Section \ref{SectionDiscussion} for more discussion on this.   

In general Question 1 is still open, and seems a difficult one. The examples in \cite{bedford-kim4} show that the usual criteria used to prove the equi-distribution property for the Green $(1,1)$ currents (see e.g. \cite{diller-guedj}, \cite{guedj3}, \cite{bayraktar}) are not applicable to a general pseudo-automorphism in dimension $3$. In fact, in the examples in \cite{bedford-kim4}, the psef eigenvector $\alpha\in H^{1,1}(X)$ with eigenvalue $\lambda _1(f)$ of $f^*:H^{1,1}(X)\rightarrow H^{1,1}(X)$ is not nef, and moreover $\alpha .C<0$ for some curve $C\subset f(I_f)$ where $I_{f}$ is the indeterminacy set of $f$. In this aspect, the following result, which provides a canonical Green $(1,1)$ current, and which whenever the equi-distribution property is satisfied is the same as the limit $T^+$ in Question 1, seems relevant. The canonical Green current $T$ constructed in Theorem \ref{TheoremGreenCurrents} is also maximal among invariant currents, in the sense that if $S$ is a positive closed $(1,1)$ current such that $f^*(S)=\lambda _1(f)S$ and $S\leq T$ then $S=cT$ for some constant $c$.

\begin{theorem}
Let $X$ be a compact K\"ahler manifold, and $f:X\rightarrow X$ a dominant meromorphic map. Let $\lambda >1$ be an eigenvalue of $f^*:H^{1,1}(X)\rightarrow H^{1,1}(X)$. Assume that the eigenvector $\theta $ corresponding to $\lambda$ is a psef class. Then there is a positive closed $(1,1)$ current $T$ such that $\{T\}=\theta$, $f^*(T)=\lambda T$, and $T$ is maximal among invariant currents. Moreover
\begin{eqnarray*}
T=\lim _{n\rightarrow\infty}\frac{1}{\lambda ^n}(f^*)^n(T_{\theta}^{\min}),
\end{eqnarray*}
where $T_{\theta}^{\min}$ is a positive closed $(1,1)$ current with minimal singularities whose cohomology class is $\theta$ (see the proof for precise definition of currents with minimal singularities). 
\label{TheoremGreenCurrents}\end{theorem}  
Note that in Theorem \ref{TheoremGreenCurrents}, we do not require that $f$ is $1$-algebraic stable or any additional condition on the eigenvector $\theta$ (such as K\"ahler, nef,...). Theorem \ref{TheoremGreenCurrents} essentially belongs to Sibony \cite{sibony}, Guedj \cite{guedj3} and Bayraktar \cite{bayraktar}. In fact, our proof is almost identical to that given for Theorem 1.2 in \cite{bayraktar} (the latter in turn followed closely that given for Theorem 2.2 in \cite{guedj3}); however it appears from the comments in those papers (e.g. Remark 4.1 after the proof of Theorem 1.2 in \cite{bayraktar}) that these authors did not aware of this. In the case $X=\mathbb{P}^k$ a projective space, Theorem \ref{TheoremGreenCurrents} was proved by Sibony \cite{sibony}. Since the theorem in the general setting has not yet appeared anywhere in the literature, we include it here for completeness. 

Part c) of Theorem \ref{TheoremInvariantCurrents} provides Green $(1,1)$ and $(2,2)$ currents $T^+$ and $T^-$  of a pseudo-automorphism with $\{T^+\}.\{T^-\}>0$. Hence, if we can make sense the wedge product of the currents $T^+$ and $T^-$, then the (signed) measure $\mu =T^+\wedge T^-$ is a good candidate for an invariant measure of $f$. Since the currents $T^{+}$ and $T^-$ are limits of Cesaro's means of currents of the form $(f^n)^*(\theta )/\lambda _1(f)^n$ and $(f^n)_*(\eta )/\lambda _1(f)^n$ for positive closed smooth $(1,1)$ and $(2,2)$ forms $\theta $ and $\eta$, the following result is relevant.   

\begin{theorem}

Let $f:X\rightarrow Y$ be a pseudo-automorphism in dimension $3$. Let $T$ be a positive closed $(1,1)$ current, and let $\eta$ be a smooth closed $(2,2)$ form. Then the intersections $T\wedge f_*(\eta )$ is well-defined with respect to Definition \ref{DefinitionIntersectionCurrents}. 

If in the above both $T$ and $\eta$ are positive then the resulting measure is also positive. If moreover $T$ has no mass on hypersurfaces then $T\wedge (f^n)_*(\eta )$ has no mass on proper analytic subvarieties of $X$, and
\begin{eqnarray*}
T\wedge (f^n)_*(\eta )=(f^n)_*((f^n)^*(T)\wedge \eta ). 
\end{eqnarray*} 
\label{TheoremIntersection11And22Currents}\end{theorem}
Remarks:

1) Since the current $(f^n)_*(\eta )$ may not be smooth on some curves where $T$ may have positive Lelong numbers, it is not obvious that we can define the wedge product $T\wedge (f^n)_*(\eta )$ intrinsically on $X$ in a reasonable way.

2) In contrast, if instead $\theta$ is a positive closed smooth $(1,1)$ form and $S$ is a positive closed $(2,2)$ current then $(f^n)^*(\theta )\wedge S$ may not be defined, or even when it can be defined the resulting may not be a positive measure. For example, consider $X$ the blowup of $\mathbb{P}^3$ at four points $[1:0:0:0]$, $[0:1:0:0]$, $[0:0:1:0]$ and $[0:0:0:1]$. Let $f:\mathbb{P}^3\rightarrow \mathbb{P}^3$ be the map $f[x_0:x_1:x_2:x_3]=[1/x_0:1/x_1:1/x_2:1/x_3]$, and let $F:X\rightarrow X$ be the lifting of $f$. Then $F$ is a pseudo-automorphism. If $C\subset X$ is the strict transform of the line $x_0=x_1=0$ and $D$ is the strict transform of the line $x_2=x_3=0$, then in cohomology $F_*\{C\}=-\{D\}$ (in \cite{truong1} it was proved that in fact the equality also holds on the level of currents: $F_*[C]=-[D]$). Now let $\theta$ be a K\"ahler form on $X$. Then even if we may define the wedge product $F^*(\theta )\wedge [C]$, the resulting current can not be a positive measure because in cohomology $F^*\{\theta\}.\{C\}=\{\theta\}.F_*\{C\}=-\{\theta\}.\{D\}<0$.   

In a forthcoming paper we will study some further properties of pseudo-automorphims in dimension $3$.

The rest of this paper is organized as follows. In Section 2 we recall definitions and results on positive and DSH currents, dynamical degrees, and known results on pseudo-automorphisms in dimension $3$. In Section 3 we prove a property of quasi-potentials of positive closed currents and a compatibility with wedge product of the kernels of Dinh and Sibony. In Section 4 we recall the definition of pullback of currents by meromorphic maps from \cite{truong1}, gives the definition of intersection of currents, and prove several general results. In Section 5 we apply the previous results to obtain results about pseudo-automorphisms in dimension $3$. In the last section we prove Theorem \ref{TheoremPositivityOfGreenCurrentsPseudoAutomorphism} and also discuss how Question 1 may be answered in affirmative. 

{\bf Acknowledgements.} The author is thankful to Eric Bedford for his kindly informing the results in \cite{bedford-cantat-kim}\cite{bedford-diller-kim}\cite{bedford-kim4} and for useful comments on a first version of this paper. He would like to thank Tien-Cuong Dinh and Viet-Anh Nguyen for some helpful conversations on this topic. He also would like to thank Keiji Oguiso, Turgay Bayraktar, Roland Roeder and \"Ozcan Yazici for their useful comments on a first version of the paper and information. Part of the work was done when the author was visiting University of Paris 6 (UPMC) and University of Paris 11 (Orsay), and he thanks these organizations for hospitality and financial support. 

\section{Preliminaries}

In this section we present briefly definitions and previous known results on positive closed and $DSH$ currents, dynamical degrees, and pseudo-automorphisms in dimension $3$. 

In this section only, let $X$ be a compact K\"ahler manifold of arbitrary dimension $k$ with a K\"ahler $(1,1)$ form $\omega _X$.

\subsection{Positive currents, $DSH$ currents}

For more details on positive currents the readers are referred to Lelong's book \cite{lelong} and Demailly's book \cite{demailly}, and for more details on $DSH$ currents the readers are referred to the paper Dinh-Sibony \cite{dinh-sibony1}. 

Given $0\leq p\leq k$, a smooth $(p,p)$ form $\varphi$ on $X$ is called strongly positive if locally it can be written as a convex combination of smooth forms of the type $i\gamma _1\wedge \overline{\gamma _1}\wedge \ldots \wedge i\gamma _p\wedge \overline{\gamma _p}$. 

A smooth $(p,p)$ form $\varphi$ is called (weakly) positive if for any strongly positive smooth $(k-p,k-p)$ form $\psi$, then $\varphi \wedge \psi$ is a positive measure.

A smooth $(p,p)$ form $\varphi$ is called strictly positive if locally $\varphi \geq \omega ^p$, where $\omega$ is a K\"ahler $(1,1)$ form.

A $(p,p)$ current $T$ is (weakly) positive if for any strongly positive smooth $(k-p,k-p)$ form $\psi$ then $T\wedge \psi$ is a positive measure.

A $(p,p)$ current $T$ is strongly positive if for any weakly positive smooth $(k-p,k-p)$ form $\psi$ then $T\wedge \psi$ is a positive measure.
 
Note that strongly and weakly positivity coincide for currents of bidegree $(0,0)$, $(1,1)$, $(k-1,k-1)$ and $(k,k)$. Therefore, if $dim(X)=3$, then strongly and weakly positivity coincide.     

For a positive $(p,p)$ current $T$, we define its mass by $||T||=<T,\omega _X^{k-p}>$. 

A current $T$ is called positive closed if it is both positive and closed. For a positive closed current $T$, its mass depends only on its cohomology class. We denote by $\mathcal{D}^p$ the real vector space generated by positive closed currents. Hence each current $T$ in $\mathcal{D}^p(X)$ can be written as $T=T^+-T^-$ for some positive closed $(p,p)$ currents $T^{\pm}$. We define the $\mathcal{D}^p$ norm of such a $T$ as follows: $||T||_{\mathcal{D}^p}=\min \{||T^+||+||T^-||\}$, where the minimum is taken on all positive closed $(p,p)$ currents $T^{\pm}$  for which $T=T^+-T^-$. We define convergence on $\mathcal{D}^p$ as follows: If $T_n$ and $T$ are in $\mathcal{D}^p$, we say that $T_n$ converges in $\mathcal{D}^p$ if $T_n$ weakly converges to $T$ in the sense of currents, and moreover $||T_n||_{\mathcal{D}^p}$ is bounded.

A $(p,p)$ current $T$ is called $DSH$ if we can find positive $(p,p)$ currents $T_1,T_2$, and positive closed $(p-1,p-1)$ currents $\Omega _1^{\pm},\Omega _2^{\pm}$ for which $T=T_1-T_2$, $dd^cT_{i}=\Omega _i^+-\Omega _i^-$ for $i=1,2$. Denote by $DSH^p(X)$ the set of $DSH$ $(p,p)$ currents on $X$. We define a norm on $DSH^p(X)$ as follows: If $T$ is in $DSH^p(X)$ then
\begin{eqnarray*}
||T||_{DSH^p}=\min \{||T_1||+||T_2||+||\Omega _1^+||+||\Omega _2^+||\},
\end{eqnarray*}   
where the minimum is taken on all decompositions $T=T_1-T_2$, $dd^cT_i=\Omega _i^+-\Omega _i^-$ of $T$. We define the convergence in $DSH$ as follows: If $T_n$ and $T$ are in $DSH^p$, we say that $T_n$ converges in $DSH^p$ if $T_n$ weakly converges to $T$ in the sense of currents, and moreover $||T_n||_{DSH^p}$ is bounded. 

\subsection{Regularization of $DSH$ currents}

In \cite{dinh-sibony1}, Dinh and Sibony obtained a good regularization of $DSH$ currents on a compact K\"ahler manifolds, which gives for any $DSH$ $(p,p)$ current $T$ a sequence of positive smooth $DSH$ currents $T_n^{\pm}$ with uniformly bounded masses so that $T_n^+-T_n^-$ weakly converges to $T$ (see also Section 3). Combining the results in \cite{truong1} and Lemma \ref{TheoremRegularizationCompactibleWithWedgproduct} in Section 3, we obtain the existence of good approximation  schemes by $C^{s}$ forms for $DSH$ currents, whose definitions are given below
\begin{definition}
Let $s\geq 0$ be an integer. We define a good approximation scheme by $C^s$ forms for $DSH$ currents on $X$ to be an assignment that for a $DSH$ current $T$ gives two sequences $\mathcal{K}_n^{\pm}(T)$ (here $n=1,2,\ldots $) where $\mathcal{K}_n^{\pm}(T)$ are $C^s$ forms of the same bidegrees as $T$, so that $\mathcal{K}_n(T)=\mathcal{K}_n^+(T)-\mathcal{K}_n^-(T)$  weakly converges to $T$, and moreover the following properties are satisfied:

1) Boundedness: If $T$ is $DSH$ then the $DSH$ norms of  $\mathcal{K}_n^{\pm}(T)$ are uniformly bounded.

2) Positivity: If $T$ is positive then $\mathcal{K}_n^{\pm}(T)$ are positive, and $||\mathcal{K}_n^{\pm}(T)||$ is uniformly bounded with respect to n.

3) Closedness: If $T$ is positive closed then $\mathcal{K}_n^{\pm}(T)$ are positive closed.  

4) Continuity: If $U\subset X$ is an open set so that $T|_U$ is a continuous form then $\mathcal{K}_n^{\pm}(T)$ converges locally uniformly on $U$.

5) Linearity: For any pair of currents $T_1$ and $T_2$, we have $\mathcal{K}_n^{\pm}(T_1+T_2)=\mathcal{K}_n^{\pm}(T_1)+\mathcal{K}_n^{\pm}(T_2)$.

6) Self-Adjointness: If $T$ and $S$ are of complement bidegrees then 
\begin{eqnarray*}
\int _X\mathcal{K}_n(T)\wedge S=\int _YT\wedge \mathcal{K}_n(S),
\end{eqnarray*}
for any $n\in \mathbb{N}$. 

7) Compatibility with the differentials: $dd^c\mathcal{K}_n^{\pm}(T)=\mathcal{K}_n^{\pm}(dd^cT)$.

8) Convergence of supports: If $A$ is compact and $U$ is an open neighborhood of $A$, then there is $n_0=n_0(U,A)$ such that if the support of $T$ is contained in $A$ and $n\geq n_0$ then support $\mathcal{K}_n(U)$ is contained in $U$.

9) Compatibility with wedge product: Let $T$ be a $DSH$ $(p,p)$ current and let $\theta$ be a continuous $(q,q)$ form on $X$. Assume that there is a positive $dd^c$-closed current $R$ so that $-R\leq T\leq R$. Then there are positive $dd^c$-closed $(p+q,p+q)$ currents $R_n$ so that $\lim _{n\rightarrow\infty}||R_n||=0$ and 
\begin{eqnarray*}
-R_n\leq \mathcal{K}_n(T\wedge \theta )-\mathcal{K}_n(T)\wedge \theta\leq R_n, 
\end{eqnarray*}
for all $n$. 

If $R$ is strongly positive or closed then we can choose $R_n$ to be so.
\label{DefinitionGoodApproximation}\end{definition}

In fact, Let $K_n$ be the weak regularization for the diagonal $\Delta _Y$ as in Section 3. Let $l$ be a large integer dependent on $s$, and let $(m_1)_n,\ldots ,(m_l)_n$ be sequences of positive integers satisfying $(m_i)_n=(m_{l+1-i})_n$ and $\lim _{n\rightarrow\infty}(m_i)_n=\infty$ for any $1\leq i\leq l$. In \cite{truong1} we showed that if we choose $\mathcal{K}_n=K_{(m_1)_n}\circ K_{(m_2)_n}\circ \ldots \circ K_{(m_l)_n}$ then it satisfies conditions 1)-8). Remark \ref{Remark1} in Section 3 shows that it also satisfies condition 9).

\subsection{Dynamical degrees and algebraic stability}

Let $f:X\rightarrow X$ be a dominant meromorphic map. It is well-known that we can define the pullback $f^*$ on smooth forms and on cohomology groups (see Section 4 for more detail). For $0\leq p\leq k=dim (X)$, the $p$-th dynamical degree of $f$ is defined by
\begin{eqnarray*}
\lambda _p(f)=\lim _{n\rightarrow\infty}||(f^n)^*(\omega _X^p)||^{1/n}.
\end{eqnarray*}
Dinh and Sibony (\cite{dinh-sibony10} and \cite{dinh-sibony1}), showed that the dynamical degrees are well-defined (i.e. the limits in the definition exist), and are bimeromorphic invariants.

Some properties of dynamical degrees: $\lambda _p(f)\geq 1$ for all $p\geq 1$, and  $\lambda _p(f)^2\geq \lambda _{p-1}(f)\lambda _{p+1}(f)$ (log-concavity). 

When $f$ is holomorphic, the results by Gromov \cite{gromov} and Yomdin \cite{yomdin} prove that the topological entropy $h_{top}(f)$ of $f$ equals $\max _{0\leq p\leq k}\log \lambda _p(f)$. For a general meromorphic map, we can still define its topological entropy. Dinh and Sibony \cite{dinh-sibony1}, proved that $h_{top}(f)\leq \max _{0\leq p\leq k}\log \lambda _p(f)$.  

Given $0\leq p\leq k$. We say that $f^*$ is $p$-algebraic stable if for any $n\in \mathbb{N}$, $(f^n)^*=(f^*)^n$ as linear maps on $H^{p,p}(X)$ (see Fornaess-Sibony \cite{fornaess-sibony1}). We can define similar notion for the pushforward $f_*$. 

\subsection{Pseudo-automorphisms in dimension $3$}
\label{SubsectionPseudoAutomorphimsDimension3}

Let now $X$ be a compact K\"ahler manifold of dimension $3$. Let $f:X\rightarrow X$ be a pseudo-automorphism with the graph $\Gamma _f\subset X\times X$. Let $\mathcal{I}(f)$ and $\mathcal{I}(f^{-1})$ be the indeterminacy sets of $f$ and $f^{-1}$. Then it follows that $f:X-\mathcal{I}(f)\rightarrow X-\mathcal{I}(f^{-1})$ is biholomorphic. Recall that $\mathcal{C}_{1}$ and $\mathcal{C}_2$ are the critical sets for the projections $\pi _1,\pi _2:\Gamma _f\rightarrow Y$, i.e. smallest analytic subsets of $\Gamma _f$ so that the restrictions $\pi _1:\Gamma _f-\mathcal{C}_1\rightarrow X$ and $\pi _2:\Gamma _f-\mathcal{C}_2\rightarrow X$ are finite-to-one maps. 
 
\begin{lemma}
a) The sets $\pi _i(\mathcal{C}_j)$ have dimensions $\leq 1$ for $i,j=1,2$.

b) For any analytic set $C$ in $X$ of dimension $\leq 1$, $f(C)$ and $f^{-1}(C)$ are analytic sets of dimensions $\leq 1$.   

c) For any $n\in \mathbb{Z}$, the maps $f^n$ are also pseudo-automorphisms.
\label{LemmaImageOfCurveByPseudoAutomorphisms}\end{lemma}
\begin{proof}

a) We prove the claim e.g. for $i=1$ and $j=2$. Since $f:X-\mathcal{I}(f)\rightarrow X-\mathcal{I}(f^{-1})$ is biholomorphic, it follows that $\pi _1(\mathcal{C}_2)$ is contained in $\mathcal{I}(f)$, and the latter has dimension $\leq 1$.

b) This also follows from the fact that $f:X-\mathcal{I}(f)\rightarrow X-\mathcal{I}(f^{-1})$ is biholomorphic. 

c) Since $f^{-1}$ is also a pseudo-automorphism, it suffices to prove the claim for $n\in \mathbb{N}$. Given $n\in \mathbb{N}$, we define $I_n(f)=\bigcup _{j=0}^n(f^{-1})^j(\mathcal{I}(f))$. By b), $I_n(f)$ is an analytic set of dimension $\leq 1$. We have $f:X-I_n(f)\rightarrow f(X-I_n(f))$ is biholomorphic, and then by induction $f^j:X-I_n(f)\rightarrow f^j(X-I_n(f))$ is biholomorphic (for $2\leq j\leq n$). The sets $f^j(X-I_n(f))$ are complements of analytic sets of dimensions $\leq 1$, by b).  Hence $f^n$ is a pseudo-automorphism. 
\end{proof}

The following result was given in Propositions 1.3 and 1.4 in \cite{bedford-kim}. For the completeness we give a proof of it here.   
\begin{lemma}
1) The maps $f_*$ and $f^*$ are all $1$- and $2$- algebraic stable. 

2) $f_*f^*=Id$ on $H^{1,1}(X)$ and $H^{2,2}(X)$. In particular, $(f^{-1})^*=f_*:H^{1,1}(X)\rightarrow H^{1,1}(X)$ is the inverse of $f^*:H^{1,1}(X)\rightarrow H^{1,1}(X)$.

3) For $\theta \in H^{1,1}(X)$ and $\eta \in H^{2,2}(X)$ then $f^*\theta .f^*\eta =\theta .\eta  $. 

\label{LemmaPseudoAutomorphismCompatibility}\end{lemma}
\begin{proof}
1) Let $\theta $ be a smooth closed $(1,1)$ form. Then $(f^n)^*(\theta )$ and $(f^*)^n(\theta )$ differ only on the set $I_n(f)=\bigcup _{j=0^n}(f^{-1})^j(\mathcal{I}(f))$. Since the latter set is analytic of dimension $\leq 1$ by Lemma \ref{LemmaImageOfCurveByPseudoAutomorphisms}, it can not contain mass for the normal current $(f^n)^*(\theta )-(f^*)^n(\theta )$. Therefore the two currents $(f^n)^*(\theta )$ and $(f^*)^n(\theta )$ are the same. Passing to cohomology we obtain that $f^*$ is $1$-stable. 

Because $f_*=(f^{-1})^*$, it follows that $f_*$ is $1$-stable. 

Since $f^*:H^{1,1}(X)\rightarrow H^{1,1}(X)$ and $f_*:H^{2,2}(X)\rightarrow H^{2,2}(X)\rightarrow H^{2,2}(X)$ are conjugates, it follows that $f_*$ is $2$-stable. Similarly, $f^*$ is $2$-stable. 

2) The proof is similar to that of 1). Let $\theta$ be a closed smooth $(1,1)$ form. Then $f_*f^*(\theta )$ and $\theta$ differ only on an analytic set of dimension $\leq 1$, and hence must be the same. Passing to cohomology we obtain the claim for $H^{1,1}(X)$. By the conjugate property and 3) we obtain the claim for $H^{2,2}(X)$.  

3) From the conjugate property and 2) we have $f^*\theta .f^*\eta =f_*f^*\theta .\eta =\theta .\eta$.  
\end{proof}

\section{Quasi-potentials and regularization kernels for DHS currents}

Let $Y$ be a compact Kahler manifold of dimension $k$. Let $\pi _1,\pi _2:Y\times Y\rightarrow Y$ be the two projections, and let $\Delta _Y\subset Y\times Y$ be the diagonal. Let $\omega _Y$ be a K\"ahler $(1,1)$ form on $Y$. As before, let $DSH^{p}(Y)$ be the space of DSH $(p,p)$ currents.

Recall that a function $\varphi$ is quasi-PSH if it is upper semi-continuous, belongs to $L^1$, and $dd^c(\varphi )=T-\theta$, where $T$ is a positive closed $(1,1)$ current and $\theta$ is a closed smooth $(1,1)$ form. We also call $\varphi$ a $\theta$-plurisubharmonic function. 
\begin{remark}
The following consideration from \cite{bost-gillet-soule} and \cite{dinh-sibony5} is used in both proof of Lemma \ref{LemmaQuasiPotential} and the construction of the kernels $K_n$ in Lemma \ref{TheoremRegularizationCompactibleWithWedgproduct}. Let $k=$ dimension of $Y$. Let $\pi :\widetilde{Y\times Y}\rightarrow Y\times Y$ be the blowup of $Y\times Y$ at $\Delta _Y$. Let $\widetilde{\Delta} _Y=\pi ^{-1}(\Delta _Y)$ be the exceptional divisor. Then there is a closed smooth $(1,1)$ form $\gamma$ and a negative quasi-plurisubharmonic function $\varphi$ so that $dd^c\varphi =[\widetilde{\Delta} _Y]-\gamma $. We choose a strictly positive closed smooth $(k-1,k-1)$ form $\eta $ so that $\pi _*([\widetilde{\Delta} _Y]\wedge \eta )=[\Delta _Y]$. 
\label{Remark2}\end{remark}

A useful tool in proving the results in Section 4 is the following, concerning the quasi-potentials of a positive closed $(p,p)$ current $T$ on a compact K\"ahler manifold $Y$. It is known that (see Dinh and Sibony\cite{dinh-sibony5}, Bost, Gillet and Soule\cite{bost-gillet-soule}) there is a $DSH$ (p-1,p-1) current $S$ and a closed smooth form $\alpha$ so that $T=\alpha +dd^cS$. Here $S$ is a difference of two negative currents. When $p=1$ or when $Y$ is a projective space, we can choose $S$ to be negative. However in general we can not choose $S$ to be negative (see \cite{bost-gillet-soule}). The following weaker conclusion is sufficient for the purpose of this paper

\begin{lemma}
Let $T$ be a positive closed $(p,p)$ current on a compact K\"ahler manifold $Y$. Then there is a closed smooth $(p,p)$ form $\alpha$ and a negative $DSH$ $(p-1,p-1)$ current $S$ so that 
\begin{eqnarray*}
T\leq \alpha +dd^cS.
\end{eqnarray*}
Moreover, there is a constant $C>0$ independent of $T$ so that $||\alpha ||_{L^{\infty}}\leq C||T||$ and $||S||\leq C||T||$. If $T$ is strongly positive then we can choose $S$ to be strongly negative.
\label{LemmaQuasiPotential}\end{lemma}   
Here $||.||_{L^{\infty}}$ is the maximum norm of a continuous form and $||.||$ is the mass of a positive or negative current.

\begin{proof} (Of Lemma \ref{LemmaQuasiPotential})
Notations are as in Remark \ref{Remark2}. Define $H=\pi _*(\varphi \eta )$. Then $H$ is a negative $(k-1,k-1)$ current on $Y\times Y$.  

We write $\gamma =\gamma ^+-\gamma ^-$ for strictly positive closed smooth $(1,1)$ forms $\gamma ^{\pm}$. If we define $\Phi ^{\pm}=\pi _*(\gamma ^{\pm}\wedge \eta )$ then $\Phi ^{\pm}$ are positive closed $(k,k)$ currents with $L^1$ coefficients. In fact (see \cite{dinh-sibony1}) $\Phi ^{\pm}$ are smooth away from the diagonal $\Delta _Y$, and the singularities of $\Phi ^{\pm}(y_1,y_2)$ and their derivatives are bounded by $|y_1-y_2| ^{-(2k-2)}$ and $|y_1-y_2|^{-(2k-1)}$. Moreover 
\begin{eqnarray*} 
dd^cH=\pi _*(dd^c\varphi \wedge \eta )=\pi _*([\widetilde{\Delta} _Y ]\wedge \eta -(\gamma ^+-\gamma ^-)\wedge \eta )=[\Delta _Y]-(\Phi ^+-\Phi ^-).
\end{eqnarray*} 
Consider $S_1=(\pi _1)_*(H\wedge \pi _2^*(T))$ and $R_1^{\pm}=(\pi _1)_*(\Phi ^{\pm}\wedge T)$. Then $S_1$ is a negative current, and $R_1^{\pm}$ are positive closed currents. Moreover 
\begin{eqnarray*}
dd^cS_1=(\pi _1)_*(dd^cH\wedge \pi _2^*(T))=T-R_1^++R_1^-.
\end{eqnarray*}
Therefore $T\leq R_1^++dd^cS_1$. Moreover $R_1^+$ is a current with $L^1$ coefficients, and there is a constant $C_1>0$ independent of $T$ so that $||S_1||, ||R_1||_{L^1}\leq C_1||T||$ (see e.g. Lemma 2.1 in \cite{dinh-sibony1}). 

If we apply this process for $R_1^+$ instead of $T$ we find a positive closed current $R_2^+$ with coefficients in $L^{1+1/(2k+2)}$ and a negative current $S_2$ so that $R_1^+\leq R_2^++dd^cS_2$. Moreover 
\begin{eqnarray*}
||R_2^+||_{L^{1+1/(2k+2)}},||S_2||\leq C_2||R_1^{+}||_{L^1}\leq C_1C_2||T||
\end{eqnarray*}
for some constant $C_2>0$ independent of $T$. After iterating this process a finite number of times we find a continuous form $R$ and a negative current $S$ so that $T\leq R+dd^cS$. Moreover, $||R||_{L^{\infty}},||S||\leq C||T||$ for some constant $C>0$ independent of $T$. Since we can bound $R$ by $\omega _Y^p$ upto a multiple constant of size $||R||_{L^{\infty}}$, we are done.   
 \end{proof}
Next we recall the construction of the kernels $K_n$ from Section 3 in \cite{dinh-sibony1}. Notations are as in Remark \ref{Remark2}. Observe that $\varphi $ is smooth out of $[\widetilde{\Delta _Y}]$, and $\varphi ^{-1}(-\infty )=\widetilde{\Delta _Y}$. Let $\chi :\mathbb{R}\cup \{-\infty\} \rightarrow \mathbb{R}$ be a smooth increasing convex function such that $\chi (x)=0$ on $[-\infty ,-1]$, $\chi (x)=x$ on $[1, +\infty ]$, and $0\leq \chi '\leq 1$.  Define $\chi _n(x)=\chi (x+n)-n$, and $\varphi _n=\chi _n\circ \varphi $. The functions $\varphi _n$ are smooth decreasing to $\varphi$, and $dd^c \varphi _n\geq -\Theta $ for every $n$, where $\Theta$ is a strictly positive closed smooth $(1,1)$ form so that $\Theta -\gamma$ is strictly positive. Then we define $\Theta _n^+=dd^c \varphi _n+\Theta $ and $\Theta _n^-=\Theta ^-=\Theta -\gamma$. Finally $K_n^{\pm}=\pi _*(\Theta _n^{\pm}\wedge \eta )$, and $K_n=K_n^{+}-K_n^-$.     

If $K$ is a current on $Y\times Y$ and $T$ a current on $Y$, we define $K(T)=(\pi _1)_*(K\wedge \pi _2^*(T))$, whenever the wedge product $K\wedge \pi _2^*(T)$ makes sense. 
\begin{lemma} Let $Y$ be a compact K\"ahler manifold. Let $K_n$ be a weak regularization of the diagonal $\Delta _Y$ defined in \cite{dinh-sibony1} (see Section 2 for more detail). Let $T$ be a $DSH$ $(p,p)$ current and let $\theta$ be a continuous $(q,q)$ form on $Y$. Assume that there is a positive $dd^c$-closed current $R$ so that $-R\leq T\leq R$. Then there are positive $dd^c$-closed $(p+q,p+q)$ currents $R_n$ so that $\lim _{n\rightarrow\infty}||R_n||=0$ and 
\begin{eqnarray*}
-R_n\leq K_n(T\wedge \theta )-K_n(T)\wedge \theta\leq R_n, 
\end{eqnarray*}
for all $n$. 

If $R$ is strongly positive or closed then we can choose $R_n$ to be so. 
\label{TheoremRegularizationCompactibleWithWedgproduct}\end{lemma}  

\begin{proof} (Of Lemma \ref{TheoremRegularizationCompactibleWithWedgproduct})

Let us define $H_n=K_n(T\wedge \theta )-K_n(T)\wedge \theta$. Since $T$ and $\theta$ may not be either positive or $dd^c$-closed, a priori $H_n$ is neither. However, we will show that there are positive $dd^c$-closed currents $R_n$ such that $\lim _{n\rightarrow\infty}||R_n||=0$ and $-R_n\leq H_n\leq R_n$. 

By definition we have
\begin{eqnarray*}
H_n(y)=\int _{z\in Y}K_n(y,z)\wedge (\pi _1^*\theta -\pi _2^*\theta )\wedge \pi _2^*T.
\end{eqnarray*}
Fix a number $\delta >0$. Then by the construction of $K_n$, there is an integer $n_{\delta}$ so that if $n\geq n_{\delta}$ and $|y-z|\geq \delta$ then $K_n(y,z)=0$. Thus
\begin{eqnarray*}
H_n(y)=\int _{z\in Y,~|z-y|<\delta}K_n(y,z)\wedge (\pi _1^*\theta -\pi _2^*\theta )\wedge \pi _2^*T.
\end{eqnarray*}
We define $h (\delta )=\max _{y,z\in Y: ~|y-z|\leq\delta}|\pi _1^*\theta -\pi _2^*\theta |$. We now show that
\begin{equation}
\lim _{\delta \rightarrow 0}h(\delta )=0.
\label{Equation.1}\end{equation} 

Let $\iota :\Delta \subset Z\times Z$ be the embedding of the diagonal $\Delta $ into $Z\times Z$. Since the $(q,q)$ form $\pi _1^*\theta -\pi _2^*\theta $ is smooth on $Z\times Z$, and since $Z\times Z$ (and hence $\Delta$) is compact, it suffices to show that the restriction of $\pi _1^*\theta -\pi _2^*\theta $ to $\Delta$ is $0$. But the latter is clear, since
\begin{eqnarray*}
\pi _1^*\theta -\pi _2^*\theta |_{\Delta}=\iota ^*(\pi _1^*\theta -\pi _2^*\theta )=(\pi _1\circ \iota )^*(\theta ) -(\pi _2\circ \iota )^*(\theta ),
\end{eqnarray*}
and the last expression is $0$ because the two maps $\pi _1\circ \iota ,\pi _2\circ \iota  :\Delta \rightarrow Z$ are the same map $(z,z)\mapsto z$. 

By (\ref{Equation.1}), because $Y\times Y$ is compact, there is a constant $C>0$ independent of $\theta$ and $\delta$ so that 
\begin{eqnarray*}
-h(\delta )C(\omega _Y(y)+\omega _Y(z))^q\leq \theta (z)-\theta (y)\leq h(\delta )C(\omega _Y(y)+\omega _Y(z))^q
\end{eqnarray*}
for all $\delta \leq 1$ and for all $|y-z|\leq \delta$. Since $K_n^{\pm}(y,z)$ are strongly positive closed and $-R\leq T\leq R$, it follows that 

\begin{eqnarray*}
H_n(y)&=&\int _{z\in Y,~|z-y|<\delta}K_n(y,z)\wedge (\pi _1^*\theta -\pi _2^*\theta )\wedge \pi _2^*T\\
&\leq&h(\delta)C\int _{z\in Y,~|z-y|<\delta}(K_n^+(y,z)+K_n^-(y,z))\wedge (\omega _Y(y)+\omega _Y(z))^q\wedge R(z)\\
&\leq&h(\delta)C\int _{z\in Y}(K_n^+(y,z)+K_n^-(y,z))\wedge (\omega _Y(y)+\omega _Y(z))^q\wedge R(z).
\end{eqnarray*}
Thus $H_n(y)\leq R_n(y)$ where 
\begin{eqnarray*}
R_n(y)=h(\delta)C\int _{z\in Y}(K_n^+(y,z)+K_n^-(y,z))\wedge (\omega _Y(y)+\omega _Y(z))^q\wedge R(z),
\end{eqnarray*}
for $n_{\delta }\leq n<n_{\delta /2} $. Similarly we have $H_n(y)\geq -R_n(y)$. It can be checked that $R_n(y)$ is positive $dd^c$-closed. Moreover, there is a constant $C_1>0$ independent of $n$, $\delta$, $R$ and $\theta$ so that 
\begin{equation}
||R_n||\leq h(\delta )C_1||R||, \label{EquationEstimateSn}
\end{equation}
for $n\geq n_{\delta}$. This shows that $||R_n||\rightarrow 0$ as $n\rightarrow \infty$.
\end{proof}
\begin{remark}
By the estimate (\ref{EquationEstimateSn}) and by iterating we obtain the following result: Let $T$, $R$ and $\theta$ be as in Lemma \ref{TheoremRegularizationCompactibleWithWedgproduct}. Then there are positive $dd^c$-closed $(p+q,p+q)$ currents $R_{n_1,n_2,\ldots ,n_l}$ so that
\begin{eqnarray*}
-R_{n_1,n_2,\ldots ,n_l}\leq K_{n_1}\circ K_{n_2}\circ \ldots K_{n_l}(T\wedge \theta )-K_{n_1}\circ K_{n_2}\circ \ldots K_{n_l}(T)\wedge \theta \leq R_{n_1,n_2,\ldots ,n_l},
\end{eqnarray*}
and 
\begin{eqnarray*}
\lim _{n_1,n_2,\ldots ,n_l\rightarrow \infty}||R_{n_1,n_2,\ldots ,n_l}||=0.
\end{eqnarray*}

We give the proof of this claim for example when $l=2$. We will write the $R_n$ in Lemma \ref{TheoremRegularizationCompactibleWithWedgproduct} by $R_n(R)$ to emphasize its dependence on $R$. Writing 
\begin{eqnarray*}
&&K_{n_1}\circ K_{n_2}(T\wedge \theta )-K_{n_1}\circ K_{n_2}(T)\wedge \theta\\
&=&[K_{n_1}(K_{n_2}(T\wedge \theta )-K_{n_2}(T)\wedge \theta )]+[K_{n_1}(K_{n_2}(T)\wedge \theta )-K_{n_1}(K_{n_2}(T))\wedge \theta ],
\end{eqnarray*}
and choosing 
\begin{eqnarray*}
R_{n_1,n_2}=K_{n_1}^+(R_{n_2}(R))+K_{n_1}^-(R_{n_2}(R))+R_{n_1}(K_{n_2}^+(R))+R_{n_1}(K_{n_2}^-(R)),
\end{eqnarray*}
we see that 
\begin{eqnarray*}
-R_{n_1,n_2}\leq K_{n_1}\circ K_{n_2}(T\wedge \theta )-K_{n_1}\circ K_{n_2}(T)\wedge \theta \leq R_{n_1,n_2}.
\end{eqnarray*}
That $R_{n_1,n_2}$ are positive $dd^c$-closed follows from the properties of the kernels $K_n$. It remains to bound the masses of $R_{n_1,n_2}$. By (\ref{EquationEstimateSn}) we have
\begin{eqnarray*}
||R_{n_1,n_2}||&\leq& C_1(||R_{n_2}(R)||+||R_{n_1}(K_{n_2}^+(R))||+||R_{n_2}(K_{n_2}^-(R))||)\\
&\leq&C_2h(\delta )(||R||+||K_{n_2}^+(R)||+||K_{n_2}^-(R)||)\\
&\leq &C_3 h(\delta )||R||,
\end{eqnarray*}
for constants $C_1,C_2,C_3$ and for all $n_1,n_2\geq n_{\delta}$, here $n_{\delta}$ is the constant in the proof of Lemma \ref{TheoremRegularizationCompactibleWithWedgproduct}.
\label{Remark1}\end{remark}

\section{Pullback of currents by meromorphic maps and intersection of currents}

\subsection{Pullback of currents}

Let $Y$ be another compact K\"ahler manifold, and let $f:X\rightarrow Y$ be a dominant meromorphic map. Let $\Gamma _f\subset X\times Y$ be the graph of $f$, and let $\pi _X,\pi _Y:X\times Y\rightarrow X,Y$ be the projections. (When $X=Y$ we denote these two maps by $\pi _1$ and $\pi _2$.) We denote by $\mathcal{C}_Y$ the critical set of $\pi _Y$, i.e. the smallest analytic subvariety
of $\Gamma _f$ so that the restriction of $\pi _Y$ to $\Gamma _f-\mathcal{C}_Y$ has fibers of dimension $dim(X)-dim(Y)$. We have a similar notation $\mathcal{C}_X$ for the map $\pi _X$. (When $X=Y$ we denote $\mathcal{C}_X,\mathcal{C}_Y$ by $\mathcal{C}_1$ and $\mathcal{C}_2$.) Hence the set $\pi _X(\mathcal{C}_Y)$ may be regarded as the critical set of the map $f$. For a set $A\subset X$, we define its (total) image by $f(A)=\pi _Y(\pi _X^{-1}(A)\cap \Gamma _f)$, and for a set $B\subset Y$ we define its (total) pre-image by $f^{-1}(B)=\pi _X(\pi _Y^{-1}(B)\cap \Gamma _f)$.

If $T$ is a smooth form on $Y$, then it is standard to define $f^*(T)$ as a current on $X$ by the formula $f^*(T)=(\pi _X)_*(\pi _Y^*(T)\wedge [\Gamma
_f])$. This definition descends to cohomology classes: If $T_1$ and $T_2$ are two closed smooth forms on $Y$ having the same cohomology classes, then
$f^*(T_1)$ and $f^*(T_2)$ have the same cohomology class in $X$. This allows us to define a pullback operator on cohomology classes. These considerations
apply equally to continuous forms. However, it is not known how to define the pullback of an arbitrary current in general. 

Meo \cite{meo} defined the pullback of a positive closed $(1,1)$ current in the following way: If $T$ is a positive closed $(1,1)$ current on $Y$, then locally we can write $T=dd^c\varphi$ where $\varphi$ is a pluri-subharmonic function, and we define $f^*(T)=dd^c(\varphi \circ f)$. There are extensions of this to the case of positive $dd^c$-closed $(1,1)$ currents (see Alessandrini- Bassanelli \cite{alessandrini-bassanelli2} and Dinh-Sibony \cite{dinh-sibony2}).

For a measure $\mu$ having no mass on the indeterminacy set $\mathcal{I}(f)$, we can define its pushforward by $f$ as follows (see e.g. \cite{deThelin-deVigny}): $(f_*)(\mu )(B)=\mu (f^{-1}(B)\cap X\backslash \mathcal{I}(f))$. 

For a holomorphic map, whose fibers are either empty or of dimension $dim(X)-dim(Y)$, Dinh-Sibony \cite{dinh-sibony2} defined pullback of positive closed currents of any bidegrees. For meromorphic selfmaps of $\mathbb{P}^k$, they gave a satisfying pullback operator using super-potentials (see \cite{dinh-sibony4}). For general meromorphic maps on compact K\"ahler manifolds, they defined a "strict pullback" on positive closed currents of any bidegrees. However, this "strict pullback" is not compatible with the pullback on cohomology. 

Using the good approximation schemes (see Definition \ref{DefinitionGoodApproximation}), we defined in \cite{truong1} a pullback operator which is compatible with the pullback on cohomology, and is compatible with the previous definitions. Moreover if a positive closed current $T$ can be pulled back by the map $f$, then $f^*(T)$ is an extension of the "strict pullback" of Dinh and Sibony. 

We now recall the definition from \cite{truong1}, where it had not been checked that the kernels $K_n$ satisfy Condition 9) in Definition \ref{DefinitionGoodApproximation}.  

\begin{definition}
 Let $T$ be a $DSH^p(Y)$ current on $Y$. We say that $f^{*}(T)$ is well-defined if there is a number $s\geq 0$ and a current $S$ on $X$ so that
\begin{eqnarray*}
\lim _{n\rightarrow\infty}f^*(\mathcal{K}_n(T))= S,
\end{eqnarray*}
for any good approximation scheme by $C^{s+2}$ forms $\mathcal{K}_n^{\pm}$. Then we write $f^{*}(T)=S$.
 \label{DefinitionPullbackCurrentsByMeromorphicMaps}\end{definition} 
The definition for a general current on $Y$ (not necessarily DSH) is more complicated. We recall it here and will use it for currents of the form $T\wedge \theta$, where $T$ is a DSH current and $\theta$ is a smooth $(q,q)$ form. Recall that since $Y$ is a compact manifold, any current on $Y$  is of finite order.
\begin{definition}
Let $T$ be a $(p,p)$ current of order $s_0$. We say that $f^{*}(T)$ is well-defined if there is a number $s\geq s_0$ and a current $S$ on $X$ so
that
\begin{eqnarray*}
\lim _{n\rightarrow}\int _{Y}T\wedge \mathcal{K}_n(f_*(\alpha )= \int _{X}S\wedge \alpha ,
\end{eqnarray*}
for any smooth form $\alpha$ on $X$ and any good approximation scheme by $C^{s+2}$ forms $\mathcal{K}_n$. Then we write
$f^{*}(T)=S$. \label{DefinitionPullbackDdcOfOrderSCurrents}\end{definition}

By the self-adjointness in Definition \ref{DefinitionGoodApproximation}, we see that Definitions \ref{DefinitionPullbackCurrentsByMeromorphicMaps} and \ref{DefinitionPullbackDdcOfOrderSCurrents} coincide for DSH currents. 

We recall some results from \cite{truong1} for using later (see Theorems 6 and 9 in \cite{truong1}):

\begin{theorem}
Let $X$ and $Y$ be two compact K\"ahler manifolds. Let $f:X\rightarrow Y$ be a dominant meromorphic map. Assume that $\pi _X(\mathcal{C}_Y)$ is of
codimension $\geq p$. Then the pullbacks $f^{*}:DSH^{p-1}(Y)\rightarrow DSH^{p-1}(X)$ and $f^*:\mathcal{D}^p(Y)\rightarrow \mathcal{D}^p(X)$ are well-defined. Moreover these pullbacks are continuous with respect to the topologies on the corresponding spaces. 
\label{TheoremInterestingExample1}\end{theorem} 
In fact, even though the statement of Theorem 6 in \cite{truong1} concerns only the claim for the map $f^*:\mathcal{D}^p(Y)\rightarrow \mathcal{D}^p(X)$ in Theorem \ref{TheoremInterestingExample1}, its proof confirms the claim for the map $f^*:DSH^{p-1}(Y)\rightarrow DSH^{p-1}(X)$ in Theorem \ref{TheoremInterestingExample1}.

\begin{theorem} Let $X$ and $Y$ be two compact K\"ahler manifolds. Let $f:X\rightarrow Y$ be a dominant meromorphic map. Let $T$ be a positive measure having no mass on $\pi _Y(\mathcal{C}_Y)$. Then $f^{*}(T)$ is well-defined, and coincides with the usual definition. Moreover, if $T$ has no mass on proper analytic
subvarieties of $Y$, then $f^{*}(T)$ has no mass on proper analytic subvarieties
of $X$.
\label{TheoremPullbackOfMeasures}\end{theorem}

We can define a similar notion $f_*$ of pushforward of currents, and obtain similar results to that of Theorems \ref{TheoremInterestingExample1} and \ref{TheoremPullbackOfMeasures} for the pushforward operator. Note that when $f$ is a bimeromorphic map then $f_*=(f^{-1})^*$.

We now prove several additional properties. A current $\tau$ is called pseudo-$dd^c$-plurisubharmonic if there is a smooth form $\gamma $ so that $dd^c\tau\geq -\gamma $. We have the following result
\begin{theorem}
Let $X$ and $Y$ be compact K\"ahler manifolds and let $f:X\rightarrow Y$ be a dominant meromorphic map. Let $T$ be a $DSH$ $(p,p)$ current and let $\theta$ be a smooth $(q,q)$ form on $Y$. Assume that there is a positive pseudo-$dd^c$-plurisubharmonic current $\tau$ so that $-\tau\leq T\leq \tau$.

a) If $f$ is holomorphic and $f^{*}(T)$ is well-defined, then $f^{*}(T\wedge \theta )$ is well-defined. Moreover, $f^{*}(T\wedge \theta )=f^{*}(T)\wedge f^*(\theta )$. 

b) More general, assume that there is a number $s\geq 0$ and a $(p,p)$ current $(\pi _Y|\Gamma _f)^{*}(T)$ on $X\times Y$  such that for any good approximation by $C^{s+2}$ forms $\mathcal{K}_n$ then 
\begin{eqnarray*} 
\lim _{n\rightarrow\infty}\pi _Y^*(\mathcal{K}_n(T))\wedge [\Gamma _f]=(\pi _Y|\Gamma _f)^{*}(T).
\end{eqnarray*}
Then $f^{*}(T\wedge \theta )$ is well-defined, and moreover $f^{*}(T\wedge \theta )=(\pi _X)_*((\pi _Y|\Gamma _f)^{*}(T)\wedge \pi _Y^*(\theta ))$.
\label{TheoremPullbackCompatibleWithWedgeProduct}\end{theorem}
Roughly speaking, the result b) of Theorem \ref{TheoremPullbackCompatibleWithWedgeProduct} says that under some natural conditions if we can pullback $T$ then we can do it locally. 

\begin{proof} (Of Theorem \ref{TheoremPullbackCompatibleWithWedgeProduct})

a) We let $s\geq 0$ be a number so that for any good approximation scheme by $C^{s+2}$ forms $\mathcal{K}_n$ and for any smooth form $\alpha$ on $X$ then 
\begin{eqnarray*}
\int _Xf^{*}(T)\wedge \alpha =\lim _{n\rightarrow\infty}\int _YT\wedge \mathcal{K}_n(f_*(\alpha )).
\end{eqnarray*}

Then for the proof of a) it suffices to show that for any smooth form $\beta$ on $X$ then

\begin{eqnarray*}
\lim _{n\rightarrow\infty}\int _YT\wedge \theta \wedge \mathcal{K}_n(f_*(\beta ))=\int _Xf^{*}(T)\wedge f^*(\theta )\wedge \beta .
\end{eqnarray*}
If we can show 
\begin{equation}
\lim _{n\rightarrow\infty}\int _YT\wedge (\theta \wedge \mathcal{K}_n(f_*(\beta ))-\mathcal{K}_n(\theta \wedge f_*(\beta )))=0
\label{Equation4}
\end{equation}
then we are done, since we have $\theta \wedge f_*(\beta ))=f_*(f^*(\theta )\wedge \beta  )$ because $f$ is holomorphic, and hence
\begin{eqnarray*}
\lim _{n\rightarrow\infty}\int _YT\wedge \mathcal{K}_n(\theta \wedge f_*(\beta ))=\lim _{n\rightarrow\infty}\int _YT\wedge \mathcal{K}_n(f_*(f^*(\theta )\wedge \beta  ))=\int _Yf^{\sharp}(T)\wedge (f^*(\theta )\wedge \beta ).
\end{eqnarray*}
Now we proceed to proving (\ref{Equation4}). For a fixed $n$ we have 
\begin{eqnarray*}
&&\int _YT\wedge (\theta \wedge \mathcal{K}_n(f_*(\beta ))-\mathcal{K}_n(\theta \wedge f_*(\beta )))\\
&=&\lim _{m\rightarrow\infty}\int _Y\mathcal{K}_m(T)\wedge (\theta \wedge \mathcal{K}_n(f_*(\beta ))-\mathcal{K}_n(\theta \wedge f_*(\beta ))).
\end{eqnarray*} 
The advantage of this is that $\mathcal{K}_m(T)$ are continuous forms, hence if we have  bounds of $\theta \wedge \mathcal{K}_n(f_*(\beta ))-\mathcal{K}_n(\theta \wedge f_*(\beta ))$ by currents of order zero we can use them in the integral and then take limit when $m\rightarrow \infty$.

Because $f_*(\beta )$ is bound by a multiple of $f_*(\omega _X^{dim (X)-p-q})$ and the latter is strongly positive closed, by condition 9) of Definition \ref{DefinitionGoodApproximation} there are strongly positive closed currents $R_n$ with $||R_n||\rightarrow 0$ and 
$$-R_n\leq \theta \wedge \mathcal{K}_n(f_*(\beta ))-\mathcal{K}_n(\theta \wedge f_*(\beta ))\leq R_n,$$ for all $n$. Since $-\tau \leq T\leq \tau $, we have $-(\mathcal{K}_m^{+}(\tau )+\mathcal{K}_m^{-}(\tau ))\leq \mathcal{K}_m(T)\leq \mathcal{K}_m^{+}(\tau )+\mathcal{K}_m^{-}(\tau )$. Since $\mathcal{K}_m^{+}(\tau )+\mathcal{K}_m^{-}(\tau )$ are positive $C^2$ forms, from the above estimates we obtain
\begin{eqnarray*}
-\int _Y(\mathcal{K}_m^+(\tau )+\mathcal{K}_m^-(\tau ))\wedge R_n&\leq& \int _Y\mathcal{K}_m(T)\wedge (\theta \wedge \mathcal{K}_n(f_*(\beta ))-\mathcal{K}_n(\theta \wedge f_*(\beta )))\\
&\leq& \int _Y(\mathcal{K}_m^+(\tau )+\mathcal{K}_m^-(\tau ))\wedge R_n.
\end{eqnarray*}
Hence (\ref{Equation4}) follows if we can show that 
\begin{eqnarray*} 
\lim _{n\rightarrow\infty}\lim _{m\rightarrow\infty}\int _Y(\mathcal{K}_m^+(\tau )+\mathcal{K}_m^-(\tau ))\wedge R_n=0.
\end{eqnarray*}
By Lemma \ref{LemmaQuasiPotential}, there are a smooth closed form $\alpha _n$ and a strongly negative current $S_n$ for which $R_n\leq \alpha _n+dd^cS_n$ and $||\alpha _n||_{L^{\infty}}, ||S_n||\rightarrow 0$. Therefore 
\begin{eqnarray*}  
0&\leq& \int _Y(\mathcal{K}_m^+(\tau )+\mathcal{K}_m^-(\tau ))\wedge R_n\\
&\leq& \int _Y(\mathcal{K}_m^+(\tau )+\mathcal{K}_m^-(\tau ))\wedge \alpha _n+\int _Y(\mathcal{K}_m^+(\tau )+\mathcal{K}_m^-(\tau ))\wedge dd^cS_n . 
\end{eqnarray*}
Since the currents $\mathcal{K}_m^{\pm}(\tau )$ are positive whose masses are uniformly bounded, it follows from $||\alpha _n||_{L^{\infty}}\rightarrow 0$ that 
\begin{eqnarray*}
\lim _{n\rightarrow\infty}\lim _{m\rightarrow\infty}\int _Y(\mathcal{K}_m^+(\tau )+\mathcal{K}_m^-(\tau ))\wedge \alpha _n=0.
\end{eqnarray*}
Now we estimate the other term. We have 
\begin{eqnarray*}
\int _Y(\mathcal{K}_m^+(\tau )+\mathcal{K}_m^-(\tau ))\wedge dd^cS_n=\int _Y(\mathcal{K}_m^+(dd^c\tau )+\mathcal{K}_m^-(dd^c\tau ))\wedge S_n.
\end{eqnarray*}
Because $S_n$ is strongly negative and $dd^c\tau \geq -\gamma $, the last integral can be bound from above by
\begin{eqnarray*}
\int _Y(\mathcal{K}_m^+(dd^c\tau )+\mathcal{K}_m^-(dd^c\tau ))\wedge S_n\leq \int _Y(\mathcal{K}_m^+(-\gamma  )+\mathcal{K}_m^-(-\gamma ))\wedge S_n.
\end{eqnarray*}
Since $\gamma $ is smooth, by condition 4) of Definition \ref{DefinitionGoodApproximation} and the fact that $||S_n||\rightarrow 0$, we obtain
\begin{eqnarray*}
\lim _{n\rightarrow\infty}\lim _{m\rightarrow\infty}\int _Y(\mathcal{K}_m^+(-\gamma  )+\mathcal{K}_m^-(-\gamma ))\wedge S_n=0.
\end{eqnarray*}
Thus, whatever the limit of 
\begin{eqnarray*}
\int _Y(\mathcal{K}_m^+(\tau )+\mathcal{K}_m^-(\tau ))\wedge dd^cS_n
\end{eqnarray*} 
is, it is non-positive. The proof of (\ref{Equation4}) and hence of a) is finished.

b) The proof of b) is similar to that of a).
\end{proof}

As some consequences we obtain the following two results, which were known previously using other definitions of pullbacks (see Diller \cite{diller}, Russakovskii-Shiffman \cite{russakovskii-shiffman} and Dinh-Sibony \cite{dinh-sibony2}). 

\begin{proposition}
Let $X$ and $Y$ be compact K\"ahler manifolds and let $f:X\rightarrow Y$ be a dominant meromorphic map. Let $\psi$ be a function on $Y$ bounded by a quasi-PSH function $\varphi$. Then $f^{*}(\varphi )$ is well-defined with respect to Definition \ref{DefinitionPullbackDdcOfOrderSCurrents}. 
\label{PropositionPullbackOfFormsWithBoundedCoefficients}
\end{proposition}
\begin{proof}
By desingularizing the graph $\Gamma _f$ if needed and using Theorem 4 in \cite{truong1}, we can assume without loss of generality that $f$ is holomorphic. By subtracting a constant from $\varphi$ if needed, we can assume that $\varphi \leq 0$. By the assumptions, we have $0\geq \psi \geq \varphi$. To prove that $f^{*}(\psi )$ is well-defined with respect to Definition \ref{DefinitionPullbackDdcOfOrderSCurrents}, we need to show the existence of a current $S$ so that for any smooth form $\alpha$ and any good approximation scheme by $C^2$ forms $\mathcal{K}_n$ then
\begin{equation}
\lim _{n\rightarrow\infty}\int _Y\psi \wedge \mathcal{K}_n(f_*(\alpha ))=\int _XS\wedge \alpha . 
\label{EquationProof1}\end{equation}
We define linear functionals $S_n$ and $S_n^{\pm}$ on top forms on $X$ by the formulas
\begin{eqnarray*}
<S_n,\alpha >&=&\int _Y\psi \wedge \mathcal{K}_n(f_*(\alpha )),\\
<S_n^{\pm},\alpha >&=&\int _Y\psi \wedge \mathcal{K}_n^{\pm}(f_*(\alpha )).
\end{eqnarray*}  
Then $S_n=S_n^+-S_n^-$, and it can be checked that $S_n^{\pm}$ are negative $(0,0)$ currents, and hence $S_n$ is a current of order $0$. Moreover, if $\alpha$ is a positive smooth measure then 
\begin{eqnarray*} 
0\geq <S_n^{\pm},\alpha >&=&\int _Y\psi \wedge \mathcal{K}_n^{\pm}(f_*(\alpha ))\\
&\geq&\int _Y\varphi \wedge \mathcal{K}_n^{\pm}(f^*(\alpha ))\\
&=&\int _Xf^*(\mathcal{K}_n^{\pm}(\varphi ))\wedge \alpha .
\end{eqnarray*}
Thus $0\geq S_n^{\pm}\geq f^*(\mathcal{K}_n^{\pm}(\varphi ))$ for all $n$.

Let us write $dd^c(\varphi )=T-\theta $ where $T$ is a positive closed $(1,1)$ current, and $\theta$ is a smooth closed $(1,1)$ form. By property 4) of Definition \ref{DefinitionGoodApproximation}, there is a strictly positive closed smooth $(1,1)$ form $\Theta$ so that $\Theta \geq \mathcal{K}_n^{\pm}(\theta )$ for any $n$. Then $f^*(\mathcal{K}_n^{\pm}(\varphi ))$ are negative $C^2$ forms so that 
\begin{eqnarray*}
dd^cf^*(\mathcal{K}_n^{\pm}(\varphi ))&=&f^*(\mathcal{K}_n^{\pm}(dd^c\varphi ))=f^*(\mathcal{K}_n^{\pm}(T-\theta ))\\
&\geq&f^*(\mathcal{K}_n^{\pm}(-\theta ))\geq -f^*(\Theta ) 
\end{eqnarray*}
for any $n$, i.e they are negative $f^*(\Theta )$-plurisubharmonic functions. Moreover the sequence of currents $f^*(\mathcal{K}_n^{\pm}(\varphi ))$ has uniformly bounded mass (see the proof of Theorem 6 in \cite{truong1}). Therefore, by the compactness of this class of functions (see Chapter 1 in \cite{demailly}), after passing to a subsequence if needed, we can assume that $f^*(\mathcal{K}_n^{\pm}(\varphi ))$ converges in $L^1$ to negative functions denoted by $f^*(\varphi ^{\pm})$. Let $S^{\pm}$ be any cluster points of $S_n^{\pm}$. Then $0\geq S^{\pm}\geq f^*(\varphi ^{\pm})$, which shows that any cluster point $S=S^+-S^-$ of $S_n$ has no mass on sets of Lebesgue measure zero. Hence to show that $S$ is uniquely defined, it suffices to show that $S$ is uniquely defined outside a proper analytic subset of $Y$. 

Let $E$ be a proper analytic subset of $Y$ so that $f:X-f^{-1}(E)\rightarrow Y-E$ is a holomorphic submersion. If $\alpha$ is a smooth measure whose support is compactly contained in $X-f^{-1}(E)$ then $f_*(\alpha )$ is a smooth measure on $Y$. Hence by condition 4) of Definition \ref{DefinitionGoodApproximation}, $\mathcal{K}_n(f_*(\alpha ))$ uniformly converges to the smooth measure $f_*(\alpha )$. Then it follows from the definition of $S$ that
\begin{eqnarray*}
<S,\alpha >=\int _{Y}\psi \wedge f_*(\alpha ).
\end{eqnarray*}
Hence $S$ is uniquely defined on $X-f^{-1}(E)$, and thus it is uniquely defined on the whole $X$, as wanted.
\end{proof}

\begin{proposition}
Let $X$ and $Y$ be compact K\"ahler manifolds and let $f:X\rightarrow Y$ be a dominant meromorphic map. Let $\pi _X,\pi _Y:X\times Y\rightarrow Y$ be the projections, and let $\Gamma _f\subset X\times Y$ be the graph of $f$. Let $V\subset Y$ be an irreducible variety. If $\pi _Y^(-1)(V)\cap \Gamma _f$ has dimension $\leq dim (V)$, then for any smooth $(q,q)$ form $\theta$ on $Y$ the pullback $f^*(\theta \wedge [V])$ is well-defined. If moreover $\theta$ is a positive form then $f^*(\theta \wedge [V])$ is also positive. 
\label{PropositionPullbackAVariety}\end{proposition}
\begin{proof}
By Lemma \ref{LemmaIntersectionOfVarieties} below we have that the intersection $\pi _Y^*([V])\wedge [\Gamma _f]$ is well-defined, and is moreover positive. Apply part b) of Theorem \ref{TheoremPullbackCompatibleWithWedgeProduct} we obtain Proposition \ref{PropositionPullbackAVariety}.
\end{proof}

\subsection{Intersection of currents}

We give the following definition of intersection of currents. It corresponds to the definition of pullback of currents for the identity map. (There are many different approaches of intersection of currents in the literature. For some references please see Bedford-Taylor \cite{bedford-taylor}, Fornaess-Sibony \cite{fornaess-sibony2}, Demailly \cite{demailly}, and Dinh-Sibony \cite{dinh-sibony4}\cite{dinh-sibony3}\cite{dinh-sibony9}.)

\begin{definition}
Let $Y$ be a compact K\"ahler manifold. Let $T_1$ be a $DSH$ current and let $T_2$ be a $(q,q)$ current of order $s$ on $Y$. Let $s_0$ be the order of $T_2$. We say that $T_1\wedge T_2$ is well-defined if there is $s\geq s_0$ and a current $S$ so that for any good approximation scheme by $C^{s+2}$ forms $\mathcal{K}_n$ then $\lim _{n\rightarrow\infty}\mathcal{K}_n(T_1)\wedge T_2=S$. Then we write $T_1\wedge T_2=S$. 
\label{DefinitionIntersectionCurrents}\end{definition}  
     
We now prove some properties of this intersection.

\begin{theorem}
Let $T_1$ and $T_2$ be positive $dd^c$-closed currents. Assume that $T_1\wedge T_2$ is well-defined. Let $\theta$ be a smooth $(q,q)$ form. 

a) $\theta \wedge T_2$ and $T_2\wedge \theta$ are well-defined and are the same as the usual definition. 

b) $T_2\wedge T_1$ is also well-defined. Moreover, $T_1\wedge T_2=T_2\wedge T_1$.

c) $T_1\wedge (\theta \wedge T_2)$ is also well-defined. Moreover $T_1\wedge (\theta \wedge T_2)=(T_1\wedge T_2)\wedge \theta$.
\label{TheoremSymmertyOfIntersectionCurrents}\end{theorem}
\begin{proof}(Of Theorem \ref{TheoremSymmertyOfIntersectionCurrents})

Proof of a): Let $\mathcal{K}_n$ be a good approximation scheme by $C^2$ forms. Then $\mathcal{K}_n(\theta )$ uniformly converges to $\theta$, and hence $\mathcal{K}_n(\theta )\wedge T_2$ converges to the usual intersection $\theta \wedge T_2$.

Let $\alpha$ be a smooth form. Then by conditions 9), 6) and 4) of Definition \ref{DefinitionGoodApproximation}, we have
\begin{eqnarray*} 
\lim _{n\rightarrow\infty}\int _Y\mathcal{K}_n(T_2)\wedge \theta \wedge \alpha&=&\lim _{n\rightarrow\infty}\int _Y\mathcal{K}_n(T_2\wedge \theta ) \wedge \alpha\\
&=&\lim _{n\rightarrow\infty}\int _YT_2\wedge \theta \wedge \mathcal{K}_n(\alpha )\\
&=&\int _YT_2\wedge \theta \wedge \alpha . 
\end{eqnarray*}

The proofs of b) and c) are similar.
\end{proof}

\begin{lemma}
Let $T_1$ and $T_2$ be positive closed $(p,p)$ and $(q,q)$ currents of $Y$. Assume that there are closed sets $A_1\subset Y$ and $A_2\subset Y$ so that $T_i$ is continuous on $Y-A_{i}$ for $i=1,2$. Assume moreover that $A_1\cap A_2$ is contained in an analytic set of codimension $\geq p+q$ of $Y$. Then $T_1\wedge T_2$ is well-defined. If moreover one of $T_1$ and $T_2$ is strongly positive then $T_1\wedge T_2$ is positive.  
\label{LemmaIntersectionOfVarieties}\end{lemma}

\begin{proof} (Of Lemma \ref{LemmaIntersectionOfVarieties})
Let $\theta$ be a smooth $(p,p)$ form having the same cohomology class as that of $T_1$. Then by Proposition 2.1 in \cite{dinh-sibony5}, there are positive $(p-1,p-1)$ currents $R^{\pm}$ so that $T_1-\theta =dd^c(R^+-R^-)$. Moreover, $R^{\pm}$ are $DSH$ and we can choose so that $R^{\pm}$ are continuous outside $A_1$. To prove Lemma \ref{LemmaIntersectionOfVarieties}, it suffices to show that there is a current $S$ so that for any good approximation scheme by $C^2$ forms $\mathcal{K}_n$ then 
\begin{eqnarray*}
\lim _{n\rightarrow\infty}\mathcal{K}_n(R^+-R^-)\wedge T_2=S.
\end{eqnarray*}
The sequence $\mathcal{K}_n^{\pm}(R^{\pm})\wedge T_2$ converges on $Y-A_1\cap A_2$. In fact, outside of $A_2$ then $T_2$ is continuous hence $\lim_{n\rightarrow\infty}\mathcal{K}_n^{\pm}(R^{\pm})\wedge T_2=R^{\pm}\wedge T_2$, and outside of $V_1$ then $\mathcal{K}_n^{\pm}(R^{\pm})$ converges locally uniformly (by condition 4) of Definition \ref{DefinitionGoodApproximation}) to a continuous form and hence $\mathcal{K}_n^{\pm}(R^{\pm})\wedge T_2$ converges. Then by an argument as in the proof of Theorem 6 in \cite{truong1} using the Federer-type support theorem in Bassanelli \cite{bassanelli}, the limit current is the trivial extension of $(R^+-R^-)|_{X-A_1\cap A_2}\wedge T_2$. In particular, we see that our definition coincides with the local definition. Since locally we can choose a local potential $H$ of $\theta$ so that the the sum of $H$ and $R^+-R^-$ gives a negative current continuous out of $A_1$ which can be well approximated by smooth negative forms whose $dd^c$ are strictly positive, the Oka's principle in \cite{fornaess-sibony2} implies that $T_1\wedge T_2$ is positive. This completes the proof of Lemma \ref{LemmaIntersectionOfVarieties}.
\end{proof}

\section{Proofs of the main results}

\subsection{Proof of Theorem \ref{TheoremPullbackPushforwardCurrents}}

That the operators $f_*,f^*$ are well-defined on spaces $\mathcal{D}^1, DSH^1$ and $\mathcal{D}^2$, and are continuous with respect to the topologies on these spaces follow from Theorem \ref{TheoremInterestingExample1} and Lemma \ref{LemmaImageOfCurveByPseudoAutomorphisms}. 

Now we show the compatibility of these operators with iterations. 

a) First we show that if $T\in DSH^{1}(X)$ then $(f^n)^*(T)=(f^*)^n(T)$ for any $n\in \mathbb{N}$. Since all the operators are continuous in the topology on $DSH^1(X)$, it suffices to prove this when $T$ is a smooth form. In this case we can proceed as in the proof of Lemma \ref{LemmaPseudoAutomorphismCompatibility}. The two currents $(f^n)^*(T)$ and $(f^*)^n(T)$ differ only on an analytic set of dimension $\leq 1$. Therefore, the current $(f^n)^*(T)-(f^*)^n(T)$ is a DSH $(1,1)$ current with support on an analytic set of dimension $\leq 1$. Since $DSH$ currents are $\mathbb{C}$-normal in the sense of Bassanelli \cite{bassanelli}, the Federer-type support theorem for $\mathbb{C}$-normal currents implies that $(f^n)^*(T)-(f^*)^n(T)$ is the zero current, i.e. $(f^n)^*(T)=(f^*)^n(T)$.

b) To extend a) to all $n\in \mathbb{Z}$ we need only to show that $(f^{-1})^*=(f^*)^{-1}$. Because $(f^{-1})^*=f_*$, it suffices to check that $f_*f^*=Id$ on $DSH^{1,1}(X)$ currents. To this end we can proceed as in a).

c) Since $\mathcal{D}^1\subset DSH ^{1,1}(X)$, we obtain the compatibility of $f^*,f_*$ for $\mathcal{D}^1$ as well. 

d) Since the operators considered are continuous on $\mathcal{D}^p$,  to prove the compatibility for $\mathcal{D}^2$, it suffices to prove the claim for smooth closed $(2,2)$ forms. Hence we need to show the following: let $\eta$ be a smooth closed $(2,2)$ form and let $\theta$ be a smooth $(1,1)$ form (not necessarily closed), then 
\begin{eqnarray*}
\int _X(f^*)^n(\eta )\wedge \theta =\int _X(f^n)^*(\eta )\wedge \theta ,
\end{eqnarray*}
for any $n\in \mathbb{N}$. By definition 
\begin{eqnarray*}
\int _X(f^n)^*(\eta )\wedge \theta =\int _X\eta \wedge (f^n)_*(\theta ),
\end{eqnarray*}
and the latter equals to
\begin{eqnarray*}
\int _X\eta \wedge (f_*)^n(\theta ),
\end{eqnarray*}
since $f_*$ is compatible with iteration on $DSH^{1,1}(X)$. Therefore, we need to show only that 
\begin{eqnarray*}
\int _X(f^*)^n(\eta )\wedge \theta =\int _X\eta \wedge (f_*)^n(\theta ),
\end{eqnarray*}
for any $n\in \mathbb{N}$.

We prove this by induction on $n$. When $n=1$, the equality follows from definition of $f^*$ and $f_*$. Assume that we already have 
\begin{eqnarray*}
\int _X(f^*)^m(\eta )\wedge \theta =\int _X\eta \wedge (f_*)^m(\theta ),
\end{eqnarray*}
for some number $m\in \mathbb{N}$. Then we will show that
\begin{eqnarray*}
\int _X(f^*)^{m+1}(\eta )\wedge \theta =\int _X\eta \wedge (f_*)^{m+1}(\theta ).
\end{eqnarray*}
 Let $\mathcal{K}_j$ be a good approximation of DSH currents by $C^2$ forms. Then $(f^*)^{m+1}(\eta )=\lim _{j\rightarrow\infty}f^*(\mathcal{K}_j(f^m)^*(\eta ))$ by the continuity of $f^*$ on $\mathcal{D}^2$. Therefore
\begin{eqnarray*} 
 \int _X(f^*)^{m+1}(\eta )\wedge \theta &=&\lim _{j\rightarrow\infty}\int _Xf^*(\mathcal{K}_j(f^*)^m(\eta ))\wedge \theta \\
 &=&\lim _{j\rightarrow\infty}\int _X\mathcal{K}_j(f^*)^m(\eta )\wedge f_*(\theta ). 
\end{eqnarray*} 
 By property 6) in Definition \ref{DefinitionGoodApproximation}, we have for any $j\in \mathbb{N}$
\begin{eqnarray*} 
 \int _X\mathcal{K}_j(f^*)^m(\eta )\wedge f_*(\theta )=\int _X(f^*)^m(\eta )\wedge \mathcal{K}_jf_*(\theta ).
\end{eqnarray*}
The currents $\mathcal{K}_jf_*(\theta )$ are $C^2$ forms by definition of $\mathcal{K}_j$, hence can be approximated uniformly by smooth $(1,1)$ forms. Therefore the induction assumption implies
\begin{eqnarray*} 
\int _X(f^*)^m(\eta )\wedge \mathcal{K}_jf_*(\theta )=\int _X\eta \wedge (f_*)^m(\mathcal{K}_jf_*(\theta ))
\end{eqnarray*}
for any $j\in \mathbb{N}$. Since $f_*$ is continuous on $DSH^{1,1}(X)$, $\lim _{j\rightarrow\infty}(f_*)^m(\mathcal{K}_jf_*(\theta ))=(f_*)^{m+1}(\theta )$. Therefore, we obtain 
\begin{eqnarray*}
\int _X(f^*)^{m+1}(\eta )\wedge \theta =\int _X\eta \wedge (f_*)^{m+1}(\theta ), 
\end{eqnarray*}
and complete the induction step, and also of Theorem \ref{TheoremPullbackPushforwardCurrents}.

\subsection{Proof of Theorem \ref{TheoremInvariantCurrents}} 

Theorem 1 in \cite{truong3} shows that under assumptions of Theorem \ref{TheoremInvariantCurrents} the growth of $||(f^n)^*_{H^{1,1}(X)}||=||(f^*)^n_{H^{1,1}(X)}||$ is $\sim \lambda _1(f)^n$.  
 
a) Let $\theta$ be a smooth closed $(1,1)$ form. We write $\theta =\theta ^+-\theta ^-$ where $\theta ^{\pm}$ are positive closed smooth $(1,1)$ forms. Since $f$ is $1$-algebraic stable, since the growth of $||(f^n)^*_{H^{1,1}(X)}||=||(f^*)^n_{H^{1,1}(X)}||$ is $\sim \lambda _1(f)^n$, it follows that there is a constant $C>0$ so that $||(f^*)^n(\theta ^{\pm})||\leq C\lambda _1(f)^n$ for any $n\in \mathbb{N}$. Hence for any $N\in \mathbb{N}$, the Cesaro's means 
\begin{eqnarray*}
T_N^{\pm}=\frac{1}{N}\sum _{j=1}^N\frac{(f^*)^j(\theta ^{\pm})}{\lambda _1(f)^j}, 
\end{eqnarray*}
are positive closed $(1,1)$ currents of mass $\leq C$. Therefore we can find a subsequence $N_j$ so that the sequences $T_{N_j}^{\pm}$ weakly converges to positive closed $(1,1)$ currents $T^{\pm}$. We define $T^{+}_{\theta }=T^+-T^-$. Then it is easy to check that $(f^*)(T^{\pm})=\lambda _1(f)T^{\pm}$, and hence $f^*(T^{+}_{\theta })=\lambda _1(f)T^{+}_{\theta }$. 

If the cohomology class $\{\theta\}\in H^{1,1}(X)$ is so that $f^*\{\theta \}=\lambda _1(f)\{\theta \}$, then $(f^*)^j\{\theta ^{+}-\theta ^{-}\}=\lambda _1(f)^j\{\theta \}$ for any $j$. Hence $\{T_N^+-T_N^-\}=\{\theta\}$ for all $N$, and therefore the cohomology class of $T^{+}_{\theta }$ is $\{\theta \}$. 

If we choose $\theta$ to be a K\"ahler form, then $T^+$ is also positive, and because the growth of $||(f^n)^*||_{H^{1,1}}$ is $\sim \lambda _1(f)^n$, $T^+$ is non-zero. In this case, we show that $T^+$ has no mass on hypersurface. This follows from the following claim:

Claim: Let $T$ be a positive closed $(1,1)$ current such that $f^*(T)=\lambda T$ for some $\lambda >1$. Then $T$ has no mass on hypersurfaces.

Proof of the claim:

This claim follows from standard arguments (see Theorem 2.4 in \cite{diller-dujardin-guedj1}). We prove by contradiction. Assume otherwise that $T$ charges hypersurfaces. Then there is a hypersurface $V$ and a number $c>0$ so that the Lelong number of $T$ along $V$ is $c$. Since $X$ is compact, by Siu's decomposition theory there is a number $M>0$ so that $\nu (T,x)\leq M$ for all $x\in X$. Let $n$ be a positive integer number so that $c>M/\lambda ^n$. Let $\mathcal{E}_f=\pi _1(\mathcal{C}_2)$ be the critical set of $f$ and let $\mathcal{I}_f$ be the indeterminacy set of $f$. By Lemma \ref{LemmaImageOfCurveByPseudoAutomorphisms}, the set $A=\{x\in X:f(x)\in\mathcal{I}_f\}\cup \{x\in X:f(x)\in\mathcal{E}_f\}$ is an analytic subset of dimension $\leq 1$. Then by results of Demailly \cite{demailly2} and Favre \cite{favre}, for all $x\in X-A$

\begin{eqnarray*}    
\nu (T,x)=\frac{1}{\lambda ^n}\nu ((f^n)^*T,x)\leq \frac{1}{\lambda ^n}\nu (T,f^n(x))\leq \frac{M}{\lambda ^n}<c. 
\end{eqnarray*}
Therefore the contradiction assumption is false, which means that $T$ has no mass on hypersurfaces. 

b) The proof of b) is similar, using Proposition \ref{Remark4} below.  

c) For any $n\in \mathbb{N}$ and any smooth closed $(2,2)$ form, we have by Lemma \ref{LemmaPseudoAutomorphismCompatibility}
\begin{eqnarray*}
\{T^+_{\theta}\}.\{(f^n)_*(\eta )\}&=&\{(f^n)^*(T^+_{\theta }).\eta \}\\
&=&\lambda _1(f)^n\{T^+_{\theta }.\eta \},  
\end{eqnarray*}
because $T^+_{\theta }$ is $f^*$ invariant. Since $T^-_{\eta}$ is a Cesaro's means of the currents $(f^n)_*(\eta )$, we have that $\{T^+_{\theta }\}.\{T^{-}_{\eta}\}=\{T^+_{\theta}\}.\{\eta\}$. Similarly we get $\{T^+_{\theta }\}.\{T^{-}_{\eta}\}=\{\theta\}.\{T^{-}_{\eta}\}$.

If we choose $\theta$ and $\eta$ to be strictly positive closed smooth forms, then $\{T^+\}.\{T^-\}>0$.

\subsection{Analytic stability}
We give here a result needed in the proof of Theorem \ref{TheoremInvariantCurrents}. This result was given as a remark without proof in \cite{truong1}. Let $X$ be a compact K\"ahler manifold of dimension $k$, and let $f:X\rightarrow X$ be a dominant meromorphic map. We recall (see Section 2) that the map $f$ is called $p$-algebraic stable if $(f^*)^n=(f^n)^*$ as linear maps on $H^{p,p}(X)$ for all $n=1,2,\ldots $. When this condition is satisfied, it follows that $\lambda _p(f)=r_p(f)$, thus helps in determining the $p$-th dynamical degree of $f$. 

There is also the related condition of $p$-analytic stable (implicitly used in \cite{dinh-sibony4} in the case $X$ is the projective space $\mathbb{P}^k$) which requires that

1) $(f^{n})^{*}(T)$ is well-defined for any positive closed $(p,p)$ current $T$ and any $n\geq 1$.

2) Moreover, $(f^n)^{*}(T)=(f^{*})^n(T)$ for any positive closed $(p,p)$ current $T$ and any $n\geq 2$.

Since $H^{p,p}(X)$ is generated by classes of positive closed smooth $(p,p)$ forms, $p$-analytic stability implies $p$-algebraic stability. For the converse of this, we have the following observation

\begin{proposition}
Let $X$ be a compact K\"ahler manifold, and $f:X\rightarrow X$ a dominant meromorphic map. If $\pi _1(\mathcal{C}_f)$ has codimension $\geq p$, then $f$ is $p$-analytic stable iff it is $p$-algebraic stable and satisfies condition 1) above so that $(f^{*})^n(\alpha )$ is positive closed for
any positive closed smooth $(p,p)$ form $\alpha$ and for any $n\geq 1$. Hence $1$-algebraic stability is the same as $1$-analytic stability.
\label{Remark4}\end{proposition}
\begin{proof}

First, let $\alpha$ be a positive closed smooth $(p,p)$ form. Then $(f^n)^*(\alpha )$ is a current with $L^1$ coefficients. Then the assumption that $(f^{*})^n(\alpha )$ is a positive closed current and the fact that $(f^{*})^n(\alpha )=(f^n)^*(\alpha )$ outside a proper analytic set imply that $(f^{*})^n(\alpha )\geq (f^n)^*(\alpha )$. But by the $p$-algebraic stability, these currents have the same cohomology class and hence must be the same. Hence the conclusion of Remark \ref{Remark4} holds for positive closed smooth $(p,p)$ forms.  

Now let $T$ be a positive closed $(p,p)$ current and let $n$ be a positive integer. By Definition \ref{DefinitionPullbackDdcOfOrderSCurrents}, there are positive closed smooth $(p,p)$ forms $T_j^{\pm}$ so that $||T_j^{\pm}||$ is uniformly bounded, $T_j^+-T_j^-$ weakly converges to $T$, and
\begin{eqnarray*}
(f^n)^{*}(T)=\lim _{j\rightarrow\infty}(f^n)^*(T_j^+-T_j^-). 
\end{eqnarray*}
By the first paragraph of the proof $(f^n)^*(T_j^+-T_j^-)=(f^{*})^n(T_j^+-T_j^-)$ for any $n$ and $j$. Because $\pi _1(\mathcal{C}_f)$ has codimension $\geq p$, the continuity property in Theorem \ref{TheoremInterestingExample1} implies that 
\begin{eqnarray*}
\lim _{j\rightarrow\infty}(f^{*})^n(T_j^+-T_j^-)=(f^{*})^n(T).
\end{eqnarray*}
Therefore $(f^n)^{*}(T)=(f^{*})^n(T)$ as wanted.
\end{proof}

\subsection{Proof of Theorem \ref{TheoremGreenCurrents}} We follow the proof of Theorem 1.2 in \cite{bayraktar}, which in turn followed closely the proof of Theorem 2.2 in \cite{guedj3}. Our proof is almost identical to that of \cite{bayraktar}, but we will include the complete proof here for convenience. We will clearly indicate in the below where our proof differs from that of \cite{bayraktar}.

First, we recall the definition of currents with minimal singularities in a psef  $(1,1)$ cohomology class. Let $\theta \in H^{1,1}(X)$  be psef, and let's choose a smooth closed $(1,1)$ form representing $\theta$, which we still denote by $\theta$ for convenience. Following Demailly-Peternell-Schneider (see the proof of Theorem 1.5 in \cite{demailly-peternell-schneider}) we define
\begin{eqnarray*}
v ^{min}_{\theta}=\sup \{\varphi \leq 0:~\theta +dd^c\varphi \geq 0\},
\end{eqnarray*}
and $T^{min}_{\theta}=\theta +dd^cv^{min}_{\theta}$. We will show that the limit 
\begin{eqnarray*}
T=\lim _{n\rightarrow\infty}\frac{1}{\lambda ^n}(f^*)^n(T^{min}_{\theta})    
\end{eqnarray*}
exists, and $T$ is what needed. 

Since $f^*\{\theta\}=\lambda \{\theta\}$ in cohomology, we have by the $dd^c$ lemma for compact K\"ahler manifolds that
\begin{eqnarray*}
\frac{1}{\lambda}f^*(T^{min}_{\theta})=\theta +dd^c\phi _1,
\end{eqnarray*}
where $\phi _1$ is a quasi-PSH function. Hence we can assume that $\phi _1\leq 0$, and from the definition of $v^{min}_{\theta}$ we get $\phi _1\leq v^{min}_{\theta}$. 

Applying $\frac{1}{\lambda}f^*$ to the above equality we find that
\begin{eqnarray*}
\frac{1}{\lambda ^2}(f^*)^2(T^{min}_{\theta})=\theta +dd^c\phi _2,
\end{eqnarray*}
where 
\begin{eqnarray*}
\phi _2=\phi _1+\frac{1}{\lambda}(\phi _1-v^{min}_{\theta})\circ f\leq \phi _1.
\end{eqnarray*}

Iterating this we obtain
\begin{eqnarray*}
\frac{1}{\lambda ^n}(f^*)^n(T^{min}_{\theta})=\theta +\phi _n,
\end{eqnarray*}
where 
\begin{eqnarray*}
\phi _n=\phi _1+\sum _{j=1}^{n-1}\frac{1}{\lambda ^j}(\phi _1-v^{min}_{\theta})\leq \phi _{n-1}.
\end{eqnarray*}
(Here is the first place where our proof differs from that in \cite{bayraktar}: We don't need $f$ to be $1$-algebraic stable here.)

$\phi _n$ is therefore a decreasing sequence of quasi-PSH functions. By Hartogs principle, either $\phi _n $ converges uniformly to $-\infty$  or converges to a quasi-PSH function $\phi$. We now use a trick by Sibony \cite{sibony} to rule out the first possibility. 

Let $R$ be a positive closed $(1,1)$ current whose cohomology class is $\{\theta\}$. We consider Cesaro's means 
\begin{eqnarray*}
R_N=\frac{1}{N}\sum _{j=1}^{N-1}\frac{1}{\lambda ^j}(f^*)^j(R).  
\end{eqnarray*}
(Here is the second place where our proof differs from that in \cite{bayraktar}: Again, we don't need $f$ to be $1$-algebraic stable.)

Notice that $R_N$ are positive closed $(1,1)$ currents having the same cohomology class $\{\theta\}$, hence have uniformly bounded masses. We can then extract a cluster point $S$. From the definition, it is easy to see that $f^*(S)=\lambda S$ and the cohomology class of $S$ is $\{\theta\}$. Therefore, by the $dd^c$ lemma we can write 
\begin{eqnarray*}
S=\theta +dd^c u,  
\end{eqnarray*}
where $u$ is a quasi-PSH function. By the invariance of $S$, after adding a constant to $u$ we can assume that
\begin{eqnarray*}
\phi _1-\frac{1}{\lambda}v_{\theta}^{min}\circ f=u-\frac{1}{\lambda}u\circ f.
\end{eqnarray*}
From this, it is easy to obtain
\begin{eqnarray*}
\phi _n=u+\frac{1}{\lambda ^n}v^{min}_{\theta}\circ f^n-\frac{1}{\lambda ^n}u\circ f^n.
\end{eqnarray*}

Here is the last and main difference between our proof and that in \cite{bayraktar}: By definition of $v^{min}_{\theta}$, there is a constant $C$ such that $u\leq v^{min}_{\theta}+C$. Therefore
\begin{eqnarray*}
\phi _n\geq u-\frac{C}{\lambda ^n}.
\end{eqnarray*}
Thus $\phi _n$ have uniformly bounded $L^1$ norms, and thus converge to a quasi-PSH function $g_{\theta}$. Therefore the limit
\begin{eqnarray*}
\lim _{n\rightarrow\infty}\frac{1}{\lambda ^n}(f^*)^n(T^{min}_{\theta} )=T
\end{eqnarray*}  
exists, where $T=\theta +dd^cg_{\theta}$. By standard arguments (see Sibony's paper \cite{sibony}), $T$ is what needed.

\subsection{Proof of Theorem \ref{TheoremIntersection11And22Currents}} Since the maps $f^n$ are all pseudo-automorphims, we need only to prove Theorem \ref{TheoremIntersection11And22Currents} for the case $n=1$. Let $T$ be a positvive closed $(1,1)$ current and let $\eta$ be a closed smooth $(2,2)$ form, we will show that $T\wedge f_*(\eta )$ is well-defined with respect to Definition \ref{DefinitionIntersectionCurrents}. We may assume without loss of generality that $\theta$ is positive. Hence  need to show that there is a $(3,3)$ current $S$ so that for any good approximation of $DSH$ currents by $C^2$ forms $\mathcal{K}_j$ then
\begin{eqnarray*}
\lim _{j\rightarrow\infty}\mathcal{K}_j(T)\wedge f_*(\eta )=S.
\end{eqnarray*}
Note that $\mathcal{K}_j(T)=\mathcal{K}_j^{+}(T)-\mathcal{K}_j^{-}(T)$, where $\mathcal{K}_j^{\pm}$  are positive closed $(1,1)$ forms of uniformly bounded masses. Define $\mu _j^{\pm}=\mathcal{K}_j^{\pm}\wedge f_*(\eta )$ then $\mu _j^{\pm}$ are positive measures of uniformly bounded masses. Therefore, there are cluster points $\mu ^{\pm}$ of $\mu _j^{\pm}$. To finish the proof of Theorem \ref{TheoremIntersection11And22Currents}, it is therefore sufficient to show that $\mu=\mu ^+-\mu ^-$ is a (signed) measure independent of the choice of the good approximation $\mathcal{K}_j$ and the subsequence defining $\mu ^{\pm}$. To this end, we will show that if $\beta$ is a smooth function on $X$ then 
\begin{equation}
<\mu ,\beta >=\int _Xf^*(\beta T)\wedge \eta .   
\label{Equation1}\end{equation}
Since $T$ is a positive closed $(1,1)$ current and $\beta$ is a smooth function, the current $\beta T$ is a DSH $(1,1)$ current. Hence by Theorem \ref{TheoremPullbackPushforwardCurrents}, the $f^*(\beta T)$ in the integral in the RHS of (\ref{Equation1}) is well-defined and is independent of either the choice of $\mathcal{K}_j$ or the subsequences defining $\mu ^{\pm}$.

We now proceed to prove (\ref{Equation1}). By definition
\begin{eqnarray*}
 <\mu ,\beta >=\lim _{j\rightarrow\infty}<\mu _j^+-\mu _j^-,\beta >=\lim _{j\rightarrow\infty}\int _X\beta\mathcal{K}_j(T)\wedge f_*(\eta ).
\end{eqnarray*}
For each $j\in \mathbb{N}$, by definition we have 
\begin{eqnarray*}
\int _X\beta\mathcal{K}_j(T)\wedge f_*(\eta )=\int _Xf^*(\beta \mathcal{K}_j(T))\wedge \eta .
\end{eqnarray*}
It is easy to check that the DSH currents $\beta \mathcal{K}_j(T)$ converges in DSH to the current $\beta T$. Hence by the continuity of $f^*:DSH^{1,1}(X)\rightarrow DSH^{1,1}(X)$, we have $\lim _{j\rightarrow\infty}f^*(\beta \mathcal{K}_j(T))=f^*(\beta T)$. Thus $T\wedge f_*(\eta )$ is well-defined. 

Note that if $\beta$ is positive then $(f^*)(\beta T)$ is positive. In fact, using desingularization of the graph of $f$ we may assume that $f$ is holomorphic. Then Theorem \ref{TheoremPullbackCompatibleWithWedgeProduct} implies that $f^*(\beta T)=f^*(\beta )f^*(T)$, and the latter is positive. Therefore the current $T\wedge f_*(\eta )$ is a positive measure. 

Assume now that moreover $T$ has no mass on hypersurfaces. Assuming the following claim, we can finish the proof. By the claim $f^*(T)$ has no mass on proper analytic subsets of $X$. Therefore the positive measure $f^*(T)\wedge \eta$ has no mass on proper analytic subsets. Thus the measure $f_*(f^*(T)\wedge \eta)$ is well-defined and has no mass on proper analytic subsets. Out of a proper analytic set $A$, the current $f_*(\eta )$ is smooth and hence the two measures $T\wedge f_*(\eta )$ and $f_*(f^*(T)\wedge \eta)$ are the same on $X-A$. Moreover these two measures have the same mass, thus they must be the same.

Claim: If $T$ has no mass on hypersurfaces then so are $f^*(T)$ and $f_*(T)$. 

Proof of the claim: The claim follows standard arguments (see e.g. Section 2.2. in Diller-Dujardin-Guedj \cite{diller-dujardin-guedj1}) and Lemma \ref{LemmaImageOfCurveByPseudoAutomorphisms}. As argued in the proof of part a) of Theorem \ref{TheoremInvariantCurrents}, there is a set $A$ which is a countable union of analytic sets of dimension $\leq 1$ so that for $x\in X-A$ we have
\begin{eqnarray*}
\nu (f^*T,x)\leq \nu (T,f(x)),
\end{eqnarray*}
where $\nu (.,.)$ is the Lelong number of a positive closed current at a point. Since $T$ has no mass on hypersurfaces, it follows by Siu's decomposition theorem that there is a set $B$ which is a countable union of analytic sets of dimension at most $1$, so that if $y\notin B$ then $\nu (T,y)=0$. Therefore $\nu (f^*T,x)=0$ for $x\in X-A-C$ where $C=\{x\in X:f(x)\notin B\}$. By Lemma \ref{LemmaImageOfCurveByPseudoAutomorphisms} again, the set $C$ is a countable union of analytic sets of dimension $\leq 1$. Therefore $f^*T$ has no mass on hypersurfaces.

The claim for $f_*(T)$ is proved similarly. 

\section{Some discussions on Question 1}
\label{SectionDiscussion}

As stated in the introduction, the usual criteria used to prove the equi-distribution property for the Green $(1,1)$ currents (see e.g. \cite{diller-guedj}, \cite{guedj3}, \cite{bayraktar}) are not applicable to the examples in \cite{bedford-kim4}. Hence a complete answer to Question 1 will require new tools developed. In this section we discuss some cases where Question 1 may be answered in affirmative. 

We first state the criteria used in \cite{diller-guedj} (see Lemma 2.5 therein) and \cite{bayraktar} (see Lemma 5.4 and Proposition 5.5 therein). The following lemma is Lemma 5.4 in \cite{bayraktar}.
\begin{lemma}
Let $f:X\rightarrow X$ be a dominant meromorphic map of a projective manifold of dimension $\geq 2$. Let $Z$ be a desingularization of the graph of $f$, and let $\pi ,g:Z\rightarrow X$ be the induced holomorphic maps (here $\pi$ is a modification). Let $\theta$ be a smooth closed $(1,1)$ form on $X$. If $\{g^*(\theta )\}.\{C\}\geq 0$ for any $\pi$-exceptional curve $C$ (i.e. a curve $C$ for which $\pi (C)$ is a point), then the potentials of $f^*(\theta )$ are bounded from above.     
\label{LemmaPositivityOfGreenCurrents}\end{lemma}

Applying this criterion, we now give a proof of Theorem \ref{TheoremPositivityOfGreenCurrentsPseudoAutomorphism}. 
\begin{proof}
Let us denote by $\theta$ a closed smooth $(1,1)$ form whose cohomology class is $\{\theta\}$. If we can show that $\{g^*(\theta )\}.\{C\}=0$ for all $\pi$-exceptional curve $C$, then Lemma \ref{LemmaPositivityOfGreenCurrents} implies that potentials of $f^*(\theta )$ are bounded from above. From the latter, it follows from the trick of Sibony that the limit
\begin{eqnarray*}
\lim _{n\rightarrow\infty}\frac{(f^*)^n(\theta )}{\lambda _1^n}  
\end{eqnarray*}
exists, and is the same for any $\theta$. We can also check easily that this limit is the same as the limit in Theorem \ref{TheoremGreenCurrents}, see e.g. Proposition 4.3 in \cite{bayraktar}.

It remains to check that for all $\pi$-exceptional curve $C$ then $\{g^*(\theta )\}.\{C\}=0$. Let $D_1,D_2,\ldots ,D_m\subset X$ be the irreducible subvarieties of dimension $dim(X)-2$ in the images of the exceptional divisors of $\pi :Z\rightarrow X$. Then by Lemma 4 in \cite{truong3}, we have in cohomology
\begin{eqnarray*}
f^*\{\theta \}.f^*\{\theta \}=\pi _*(g^*\{\theta\}).\pi _*(g^*\{\theta\})=\pi _*(g^*(\{\theta \}.\{\theta \}))+\sum _jl(g^*(\theta ),D_j)\{D_j\} .
\end{eqnarray*}
Here $l(g^*(\theta ),C_j)$ are non-negative numbers depending on the intersections between $g^*(\theta )$ and the $\pi$-exceptional curves belonging to $\pi ^{-1}(D_j)$. Since $f^*(\{\theta\})=\lambda _1\{\theta\}$, the assumption (i) that $\{\theta \}.\{\theta \}=0$ implies that $\sum _jl(g^*(\theta ),D_j)\{D_j\}  =0$ which means that each individual term $l(g^*(\theta ),D_j)=0$. The latter, when combined with the proof of Lemma 4 in \cite{truong3} and our assumption (ii), implies that $\{g^*\theta\}.\{C\}=0$ for all $\pi$-exceptional curve $C$. The proof is completed. 
\end{proof}  

There are several other classes of pseudo-automorphisms in dimension $3$ where we expect that Question 1 has an affirmative answer. Let $f:X\rightarrow X$ be a pseudo-automorphism in dimension $3$ such that $\lambda _1(f)^2>\lambda _2(f)$. Let $T^+$ be a non-zero positive closed $(1,1)$ current such that $f^*(T^+)=\lambda _1(f)T^+$. In fact, it seems that Case 2 below should be true for any such pseudo-automorphism. Let $Z$ be a desingularization of the graph of $f$ and let $\pi ,g:Z\rightarrow X$ be the induced maps, where $\pi :Z\rightarrow X$ is a finite composition of blowups along smooth centers.  

Case 1: $\{T^+\}.\{T^+\}=0$. This condition is the same as condition (i) in Theorem \ref{TheoremPositivityOfGreenCurrentsPseudoAutomorphism}, and we think that condition (i) alone is enough to prove the conclusions of Theorem \ref{TheoremPositivityOfGreenCurrentsPseudoAutomorphism}.

Case 2: There is a K\"ahler form $\omega$ such that the cluster points of the sequence
\begin{eqnarray*}
\frac{1}{\lambda _1(f)^n}f^*((f^n)^*(\omega ))\wedge f^*(\omega )-\frac{1}{\lambda _1(f)^n}f^*((f^n)^*(\omega )\wedge \omega )
\end{eqnarray*}
contain a positive closed current. In this case taking limit when $n\rightarrow\infty$ we have by Lemma 4 in \cite{truong3}
\begin{eqnarray*}
f^*\{T^+\}.f^*\{\omega \}-f^*\{T^+\wedge \omega \}=\sum _j l(g^*T^+,g^*(\omega ),D_j)\{D_j\},
\end{eqnarray*}
where $l(g^*T^+,g^*(\omega ),D_j)$ are non-negative numbers depending bilinearly on $g^*T^+$ and $g^*(\omega )$. More precisely, $l(g^*T^+,g^*(\omega ),D_j)$ depends on the products of numbers $\{g^*T^+\}.\{C\}$ and $\{g^*(\omega )\}.\{C\}$ where $C$ are $\pi$-exceptional curves belonging to $\pi ^{-1}(D_j)$. Since $\{g^*(\omega )\}.\{C\}> 0$ for all curves $C$ for which $g(C)$ is not a point, we expect that $\{g^*T^+\}.\{C\}\geq 0$ for all $\pi$-exceptional curves. The expectation is true in the simplest case where $\pi :Z\rightarrow X$ is a blowup along a finite number of pairwise disjoint curves in $X$.


\end{document}